\documentclass{amsart}
\usepackage{amssymb}
\usepackage{amscd}
\usepackage{amsfonts}
\usepackage{latexsym}
\usepackage[all]{xy}
\usepackage{verbatim}

\newtheorem{theorem}{Theorem}[section]
\newtheorem{lemma}[theorem]{Lemma}
\newtheorem{proposition}[theorem]{Proposition}
\newtheorem{corollary}[theorem]{Corollary} 
\theoremstyle{definition}  
\newtheorem{definition}[theorem]{Definition}
\newtheorem{example}[theorem]{Example}
\newtheorem{conjecture}[theorem]{Conjecture}  

\newtheorem{remark}[theorem]{Remark}
\newtheorem{note}[theorem]{Note}

\newcommand{\id}{\text{id}}

\newcommand{\Fun}{\text{Fun}}

\renewcommand{\Vec}{\text{Vec}}
\newcommand{\Pic}{\text{Pic}}
\newcommand{\BrPic}{\text{BrPic}}

\newcommand{\Hom}{\text{Hom}}

\newcommand{\Aut}{\text{Aut}}

\newcommand{\Rep}{\text{Rep}}

\newcommand{\rev}{\text{rev}}

\newcommand{\op}{\text{op}}

\newcommand{\B}{\mathcal{B}}
\newcommand{\C}{\mathcal{C}}
\newcommand{\D}{\mathcal{D}}

\newcommand{\Z}{\mathcal{Z}}

\newcommand{\M}{\mathcal{M}}
\newcommand{\A}{\mathcal{A}}
\newcommand{\N}{\mathcal{N}}

\newcommand{\be}{\mathbf{1}}

\newcommand{\la}{\langle\,} 
\newcommand{\ra}{\,\rangle}

\renewcommand{\be}{\mathbf{1}}

\newcommand{\cC}{{\mathcal C}}

\newcommand{\bt}{\boxtimes}
\newcommand{\ot}{\otimes}

\newcommand{\beq}{\begin{equation}}
\newcommand{\eeq}{\end{equation}}

\newcommand{\lb}{\label}
\newcommand{\bpf}{\begin{proof}}
\newcommand{\epf}{\end{proof}}

\newcommand{\bth}{\begin{theorem}}
\renewcommand{\eth}{\end{theorem}}
\newcommand{\bpr}{\begin{proposition}}
\newcommand{\epr}{\end{proposition}}
\newcommand{\ble}{\begin{lemma}}
\newcommand{\ele}{\end{lemma}}
\newcommand{\bco}{\begin{corollary}}
\newcommand{\eco}{\end{corollary}}
\newcommand{\bde}{\begin{definition}}
\newcommand{\ede}{\end{definition}}
\newcommand{\bex}{\begin{example}}
\newcommand{\eex}{\end{example}}
\newcommand{\bre}{\begin{remark}}
\newcommand{\ere}{\end{remark}}
\newcommand{\bcj}{\begin{conjecture}}
\newcommand{\ecj}{\end{conjecture}}

\newcommand{\End}{\text{End}}

\newcommand{\opp}{\text{op}}

\begin{document}
\title{The Picard crossed module of a braided tensor category}
\author{Alexei Davydov}
\address{A.D.: Department of Mathematics and Statistics,
University of New Hampshire,  Durham, NH 03824, USA}
\email{alexei1davydov@gmail.com}
\author{Dmitri Nikshych}
\address{D.N.: Department of Mathematics and Statistics,
University of New Hampshire,  Durham, NH 03824, USA}
\email{nikshych@math.unh.edu}

\date{\today}
\begin{abstract}
For a finite braided tensor category $\C$ we introduce its {\em Picard crossed module} $\mathfrak{P}(\C)$  consisting of
the group of invertible $\C$-module categories and the group of braided tensor autoequivalences of $\C$.
We describe $\mathfrak{P}(\C)$ in terms of braided autoequivalences of the Drinfeld center of $\C$. 
As an illustration, we compute the Picard crossed module of a braided pointed fusion category.
\end{abstract} 

\maketitle

\section{Introduction}

Tensor categories can be thought of as categorical analogues  of associative algebras.
One  can adapt   standard notions and constructions of the classical theory of associative algebras
to tensor categories.  Analogues of (bi-)modules over algebras are {\em (bi-)module categories}
over tensor categories \cite{qu, jk, O1}.

Given an algebra $C$  the isomorphism classes of  invertible $C$-bimodules
form a group $\BrPic(C)$ called the {\em Brauer-Picard} group of $C$.  There is a well known
homomorphism
\begin{equation}
\label{classical phi}
\phi: \BrPic(\C)\to \Aut(Z(C)),
\end{equation}
where $Z(C)$ denotes the center of $C$, constructed  as follows. Given an invertible $C$-bimodule $M$
and $z\in Z(C)$,  the element $\phi(M)(z)\in Z(C)$ is defined by the condition that the endomorphism 
of $M$ given by the left multiplication
by $\phi(M)(z)$ equals to that given by the right multiplication by $z$. 

There is an analogue of homomorphism \eqref{classical phi} for tensor categories.  Given a finite tensor category $\C$
one defines its  {\em Brauer-Picard} group $\BrPic(\C)$  of equivalence classes of invertible $\C$-bimodule categories 
(see \cite{ENO}) and 
a homomorphism
\begin{equation}
\label{categorical Phi}
\Phi: \BrPic(\C) \to  \Aut^{br}(\Z(\C)),
\end{equation}
where $\Z(\C)$ is the {\em Drinfeld center} of $\C$ and $\Aut^{br}(\Z(\C))$ is the group of braided
autoequivalences of $\Z(\C)$. 

It was shown in \cite{ENO} that \eqref{categorical Phi} is an isomorphism when $\C$  is a fusion category. 

Braided tensor categories are analogues of  commutative algebras.  Similarly to the classical case,
module categories over a braided tensor category $\C$ can be regarded as bimodule categories. 
In this case the group $\BrPic(\C)$ contains a subgroup $\Pic(\C)$, called the {\em Picard group} of $\C$,
consisting of invertible $\C$-module categories \cite{ENO}. One defines a homomorphism
\begin{equation}
\label{categorical partial}
\partial: \Pic(\C) \to  \Aut^{br}(\C),
\end{equation}
in a way parallel to \eqref{categorical Phi}.  Note that the classical analogue of
\eqref{categorical partial} for commutative algebras is trivial.  But, in general, $\partial$ 
is far from being trivial. It was shown in \cite{ENO} that it is an isomorphism for every non-degenerate 
braided fusion category~$\C$. 

Groups $\Pic(\C)$ and $\Aut^{br}(\C)$ play important role in the theory of braided tensor categories.
In particular, they are used in the classification of group extensions of fusion categories \cite{ENO}.
They also appear as parts of an important invariant of $\C$ called the {\em core}  studied
in \cite{DGNO}.  We thus hope that our description of the algebraic structure formed by these
groups will shed more light on the above constructions.

Below  is the summary of our results. 

The starting point of this paper
is a conjecture of V.~Drinfeld that for a braided tensor category $\C$
the pair $\mathfrak{P}(\C)= (\Pic(\C),\, \Aut^{br}(\C))$ along with the homomorphism 
\eqref{categorical partial} and the natural action   of $\Aut^{br}(\C)$ on $\Pic(\C)$ is a {\em crossed module}, called the 
{\em Picard crossed module} of $\C$.  See Section~\ref{def Pic cr mod}
for the definition of a crossed module and \cite{JS}, \cite[Appendix E.5.3]{DGNO} for an interpretation
of crossed modules in terms of monoidal categories. We prove this conjecture in Theorem~\ref{crmodthm}.

For a finite tensor category  $\C$ 
we define its Brauer-Picard group $\BrPic(\C)$  as the group of equivalence classes of invertible {\em exact}
$\C$-bimodule  categories.
We prove in Theorem~\ref{BrPic=Autbr theorem} that the canonical homomorphism \eqref{categorical Phi}
is an isomorphism. This  extends the corresponding result for fusion categories proved in \cite{ENO}. 

Next, for a braided finite tensor category $\C$ we show in Theorem~\ref{Drinfeld's Theorem} that the
image of $\Pic(\C)\subset \BrPic(\C)$ under the isomorphism
\eqref{categorical Phi} is the subgroup of braided autoequivalences of $\Z(\C)$
trivializable on $\C$.   

Finally, we  explicitly compute 
the Picard crossed module of a pointed braided fusion category in Section~\ref{pointed example}.
It turns out that the Picard groups of pointed braided fusion categories interpolate
between  the orthogonal groups of quadratic forms and the exterior squares of finite Abelian groups.  

The paper is organized as follows.

Section~\ref{prelim} contains basic facts about finite tensor categories and module categories over them. 
Here we also define  the Brauer-Picard group
of a finite tensor category  and the Picard group of a finite braided tensor category. (They were previously
defined in \cite{ENO} in the setting of fusion categories).

In Section~\ref{Picard section}  we introduce the Picard crossed module of a braided tensor category.

In Section~\ref{section braided aut} we prove our main Theorems~\ref{BrPic=Autbr theorem} and
\ref{Drinfeld's Theorem} and describe the Picard crossed module  of a braided tensor category in terms of braided autoequivalences
of its center.

Section~\ref{pointed example} is devoted to the computation of  the Picard crossed module of a pointed braided 
fusion category and its invariants.

\textbf{Acknowledgment.}
We are deeply grateful V. Drinfeld. The statements of Theorems~\ref{crmodthm} and
\ref{Drinfeld's Theorem}  are due to him.  We also thank  J.~Cuadra and V.~Ostrik for valuable comments.
The work of the second named author  was partially supported  by the NSF grant DMS-0800545.

\section{Prelimimaries}
\label{prelim}

\subsection{General conventions}

We work over  an algebraically closed field $k$ of characteristic zero.
Recall that a $k$-linear abelian category $\C$ is {\em finite} if 
 \begin{enumerate}
\item[(i)] $\C$ has finite dimensional spaces of morphisms;
\item[(ii)] every object of $\C$ has finite length;
\item[(iii)] $\C$ has {\em enough projectives}, i.e., every simple object of $\C$ has a projective cover; and
\item[(iv)] there are finitely many isomorphism classes of simple objects in $\C$.
\end{enumerate}
All abelian categories considered in this paper will be  finite.
Any such category  is equivalent to the category $\Rep(A)$ of finite dimensional representations of  
a finite dimensional  $k$-algebra $A$.  All functors between such categories will be additive and $k$-linear.
We use 
the symbol $\simeq$ for equivalence between categories  and the symbol $\cong$ for isomorphisms between objects.

In this paper we freely use basic results of the theory of finite tensor categories and module categories
over them \cite{BK, EO, O1} and  the theory of braided  categories \cite{JS, DGNO}. 

\subsection{Tensor categories}

By a {\em tensor category} we mean a  finite rigid tensor category $\A$ whose  unit object $\be$ is simple \cite{EO}. 
A semisimple tensor category is called a {\em fusion} category.

Let $\A$ be a tensor category with the associativity constraint 
\[
a_{X,Y,Z}:(X\ot Y)\ot Z \xrightarrow{\simeq} X \ot (Y \ot Z). 
\]
The tensor category with the opposite tensor product
$X\otimes^{\opp}Y := Y\otimes X$ and the accordingly adjusted associativity constrain $a^{\opp}$:
$$
\xymatrix{(X\otimes^{\opp}Y)\otimes^{\opp}Z \ar@{=}[d] \ar[rr]^{a^{\opp}_{X,Y,Z}} && X\otimes^{\opp}(Y\otimes^{\opp}Z)  \ar@{=}[d] \\
Z\otimes (Y\ot X) \ar[rr]^{a_{Z,Y,X}^{-1}} && (Z\ot Y)\ot X}
$$
will be called the category {\em opposite} to $\A$ and will be denoted $\A^{\opp}$. 

Let $\A$ and $\B$ be tensor categories. Their  {\em Deligne} tensor product \cite{D} will be denoted $\A\boxtimes\B$. 

\begin{definition}
\label{trivializable def}
Let $\A$ be a tensor category and let $\mathcal{B}\subset \A$ be a tensor subcategory.
A tensor autoequivalence $\alpha$ of $\A$   is called {\em trivializable} on $\mathcal{B}$ if the restriction
$\alpha |_\mathcal{B}$ is isomorphic to $\id_\mathcal{B}$ as a tensor functor.
\end{definition}

We will denote by $\Aut(\A)$ (respectively, $\Aut(\A,\, \mathcal{B})$) the group of isomorphism classes 
of tensor autoequivalences of $\A$ (respectively, tensor autoequivalences of $\A$ trivializable on $\mathcal{B}$).

\subsection{Braided tensor categories}

Recall that a {\em braided} tensor category $\C$ is a finite tensor category equipped with a natural 
isomorphism 
\[
c_{X,Y}: X \ot Y \xrightarrow{\simeq} Y \ot X
\] 
satisfying the hexagon axioms \cite{JS}. 
The braiding of $\C$ gives rise to a tensor  equivalence between  $\C$ and $\C^{\opp}$.

An important example of a braided tensor category  is the {\em center} $\Z(\A)$ of a finite tensor category $\A$.
It is defined as the category whose objects  are pairs $(Z,\, \gamma)$, where $X$ is an object of $\A$ and $\gamma$
is a natural family of isomorphisms, called {\em half braidings}:
\[
\gamma_{X}: X \ot Z \xrightarrow{\simeq}  Z \ot X,\qquad X\in \A,
\]
satisfying compatibility conditions.  The center is a finite braided tensor category
with the braiding given by
\[
\delta_{Z} : (Z,\, \gamma) \ot (Y,\, \delta) \xrightarrow{\simeq}  (Y,\, \delta)  \ot (Z,\, \gamma). 
\]

Let $\C^\rev$ denote the  tensor category  $\C$ equipped with  the reversed braiding
\[
\tilde{c}_{X,Y} =c_{Y,X}^{-1}.
\] 
For a braided tensor category $\C$ there are canonical embeddings $\C \hookrightarrow \Z(\C)$
and $\C^\rev  \hookrightarrow \Z(\C)$
given by 
\begin{equation}
\label{embeddings of C to Z(C)} 
X \mapsto  (X,\, c_{-,X}) \qquad \mbox{and} \qquad X \mapsto (X,\, \tilde{c}_{-,X}). 
\end{equation}

For a braided tensor category $\C$  the embeddings \eqref{embeddings of C to Z(C)}  combine
into a single braided tensor functor
\begin{equation}
\label{functor G} \C \bt \C^\rev \to \Z(\C).
\end{equation}
A braided tensor category $\C$ is called {\em factorizable} if the functor \eqref{functor G} is an equivalence. 

We will denote by $\Aut^{br}(\C)$ the group of isomorphism classes of braided tensor autoequivalences of a braided tensor category $\C$. 

\begin{example}
\label{quad form}
Let $\C$ be a pointed braided fusion category. Then isomorphism classes of  simple
objects of $\C$ form a finite Abelian group $A$. 

The associativity constraint of $\C$ determines a  $3$-cocycle $\omega: A\times A \times A \to k^\times$.
The braiding determines a function 
\begin{equation}
\label{function c}
c:  A\times A \to k^\times
\end{equation}
satisfying the following identities
coming from the hexagon axioms of braided tensor category:
\begin{eqnarray}
\label{ab cocycle A}
c(x,\, y+z)c(x,\,y)^{-1}c(x,\,z)^{-1} &=&
\omega(x,\,y,\,z)\omega(y,\,x,\,z)^{-1}\omega(y,\,z,\,x), \\
\label{ab cocycle B}
c(x+y,\,z)c(x,\,z)^{-1}c(y,\,z)^{-1} &=&
\omega(x,\,y,\,z)^{-1}\omega(x,\,z,\,y)\omega(z,\,x,\,y)^{-1},
\end{eqnarray}
for all $x,\,y,\,z\in A$. 
Following \cite{em} denote the set of pairs $(\omega,c)$ satisfying the equations (\ref{ab cocycle A}) 
by $Z^3_{ab}(A,\,k^\times)$. Note that $Z^3_{ab}(A,\,k^\times)$ is a group with respect to pointwise multiplication.

A tensor functor $F:\cC\to \cC'$ gives rise to a group homomorphism $f:A\to A'$. The tensor structure of $F$
gives rise to a map $\phi: A \times A\to k^\times$. The coherence axiom for the tensor  structure 
becomes the  2-coboudary condition:
\begin{equation}\label{cob}
\phi(y,\,z)\phi(x+y,\,z)^{-1}\phi(x,\,y+z)^{-1}\phi(x,\,y)^{-1} = \omega(x,\,y,\,z)\omega'(f(x),f(y),f(z))^{-1},  
\end{equation}
for all $x,\,y,\,z\in A$. Here
$\omega,\,  \omega'$ are the associativity constraints in $\cC,\, \cC'$ respectively. The tensor functor $F$ is braided if
\begin{equation}\label{bra}
c(x,\,y)c'(f(x),\,f(y))^{-1} = \phi(x,\,y)\phi(y,\,x)^{-1}.
\end{equation}

Tensor autoequivalences isomorphic to the identity functor (identity $f$) define an equivalence relation on the group
of pairs $(\omega,c)$, where $(\omega,c)$ and $(\omega',c')$ are related as in (\ref{cob},\ref{bra}) with trivial $f$. The
quotient group is known as the {\em third abelian cohomology} $H^3_{ab}(A,\,k^\times)$ \cite{em}.

The function 
\[
q(x):= c(x,\,x),\quad x\in A,
\] is a {\em quadratic form} on $A$, i.e., $q(-x)=q(x)$ and
the symmetric function 
\begin{equation}
\label{bimult sigma}
\sigma(x,\,y) = \frac{q(x+y)}{q(a)q(b)},\qquad  x, y\in A
\end{equation}
is bimultiplicative. We have the identity
\begin{equation}
\label{cc = sigma}
\sigma(x,\,y) =  c(x,\,y) \, c(y,\,x),\qquad x,\,y\in A. 
\end{equation}

It was proved in \cite{em} that the map $(\omega,c)\mapsto q$ defines an isomorphism between  $H^3_{ab}(A,k^*)$ and the group of quadratic forms $A\to k^\times$.

By associating to $\C$ the pair
$(A,\, q)$ one gets a functor from the 1-categorical contraction of the 2-category of pointed
braided fusion categories to the category of {\em pre-metric groups}. Objects of the latter category are 
finite abelian groups equipped with a quadratic forms and morphisms are group homomorphisms
preserving the quadratic forms (i.e., orthogonal homomorphisms).

It was proved by Joyal and Street 
\cite{JS} that the above functor is an equivalence (see also \cite[Appendix D]{DGNO}). 
The braided fusion category associated to $(A,\, q)$ will be denoted $\C(A,\, q)$. 

It follows from the above that
\[
\Aut^{br}(\C(A,\, q)) = O(A,\, q),
\]
where $O(A,\,q)$ denotes the group of orthogonal  automorphisms of $(A,\,q)$, 
i.e., auotomorphisms $\alpha: A\to A$ such that $q\circ \alpha = q$.

\end{example}

\subsection{Centralizers in braided tensor categories}
\label{Prelim centralizers}

In \cite{M1} M.~M\"uger introduced the following definition.

\begin{definition}
Objects $X$ and $Y$ of a braided tensor category $\C$  are said to {\em centralize} each other if
\[
c_{Y,X} c_{X,Y} =\id_{X \ot Y}. 
\]
The {\em centralizer} $\D'$ of a tensor subcategory $\D \subset \C$ is defined to be the full subcategory
of objects of $\C$ that centralize each object of $\D$. It is easy to see that $\D'$ is a tensor subcategory. 
\end{definition}

We will denote the self-centralizer $\C'$ of $\C$ by $\Z_{sym}(\C)$ and call it  the {\em symmetric}  center of $\C$.
We say that $\C$ is {\em non-degenerate}  if and only if $\Z_{sym}(\C)$ is trivial, i.e., consists of extensions of the unit
object $\be$. 

\begin{remark}
It was shown in \cite[Proposition 3.7]{DGNO} that a braided tensor category $\C$ is non-degenerate
if and only it is factorizable. 
\end{remark}

Let $\C$ be a braided tensor category. Let us identify $\C$ and $\C^\rev$ with their images in $\Z(\C)$
under the embeddings \eqref {embeddings of C to Z(C)}. Then $\C$ and  $\C^\rev$ are centralizers of each other.  

\begin{example}
\label{centralizers in C(A,q)}
Let us describe the centralizers in the pointed braided fusion category $\C(A,\, q)$, see Example~\ref{quad form}.
Two simple objects $x,\, y\in A$ of this category centralize each other if and only if $\sigma(x,\, y)=1$, where
$\sigma$ is the bimultiplicative symmetric function  \eqref{bimult sigma} corresponding to $q$. That is,
in this case the centralizing property coincides with orthogonality. 

Every fusion subcategory of $\C(A,\, q)$ corresponds to a subgroup $B \subset A$ and is equivalent to $\C(B,\, q|_B)$.
We have $\C(B,\, q|_B)' =\C(B^\perp,\, q_|{B^\perp})$, where $B^\perp$ is the subgroup of $A$ orthogonal to $B$. 
In particular,
\[
\Z_{sym}(\C(A,\,q)) = \C(A^\perp,\,q|_{A^\perp}),
\]
where  $A^\perp = \{a\in A|\ \sigma(a,b)=1\ \forall b\in A\}$ the kernel of $\sigma$. 
The category $\C(A,\, q)$ is non-degenerate if and only if $\sigma$ is non-degenerate. 
\end{example}

\subsection{Module categories over tensor categories}
\label{mcmc}

Let $\A$ be a finite tensor category.
A left {\em $\A$-module} category (see \cite{qu,jk,O1}) is a finite category $\M$ 
together with a bifunctor  
$$
\A\times\M\to\M,\quad (X,M)\mapsto X*M
$$
equipped with  a functorial  isomorphism called the {\em associativity  constraint}:
\[
a_{X,Y,M}: X * (Y * M) \xrightarrow{\simeq} (X \ot Y)* M,\qquad X,\,Y\in \A,\, M \in \M,
\]
and the unit constraint satisfying natural compatibility axioms. 

Equivalently,  $\M$ is a left module category over $\A$ if there is given a tensor functor $\A\to\End(\M)$ 
to the tensor category $\End(\M)$ of endofunctors of $\M$ (with tensor structure given by composition of functors). 

A right $\A$-module category is defined in a similar way. It corresponds to a tensor functor $\A^{\opp}\to\End(\M)$. 
For a right $\A$-module category $\M$ the category  
obtained from $\M$ reversing the directions of morphisms is a left $\A$-module category  via
\[
X \odot M = M*X^* ,\qquad M\in \M,\, X\in \A. 
\]
We will denote this category $\M^{\opp}$ and call it the {\em opposite} module category. 

Functors between $\A$-module categories and natural transformations between them are defined in  
an obvious way, see \cite{O1}.  

Let $\A$ be a tensor category. Following \cite{EO} we say that an $\A$-module category $\M$ is {\em exact}
if for any projective object $P$ of $\A$ and every object $M$ of $\M$ the object $P\ot M$ is projective. 
An $\A$-module category $\M$ is exact if and only if for every $\C$-module category $\N$ any $\C$-module  
functor  $\M\to \N$ is exact.

\begin{example}
If $\A$ is  a fusion category then an $\A$-module category is exact if and only if it is semisimple. 
\end{example}

\begin{note}
All module categories in this paper are assumed to be exact. 
\end{note}

Given an indecomposable  left $\A$-module category $\M$ the {\em dual} category of $\A$ with respect to $\M$
is the category $\A^*_\M=\Fun_\A(\M,\, \M)$ of $\A$-module endofunctors of $\M$. It was shown in
\cite[Section 3.3]{EO} that  $\A^*_\M$ is a finite tensor category.  Furthermore, $\M$ is an exact indecomposable 
left $\A^*_\M$-module  category and there is a canonical tensor equivalence $\A \cong (\A^*_\M)^*_\M$.

\begin{remark}
\label{great EO}
It was proved in \cite[Theorem 3.31]{EO} that
the assignment 
\[
\N \mapsto  \Fun_\A(\M,\, \N)
\]
is an equivalence between the $2$-category of
exact left $\A$-module categories  and that of exact right $\A^*_\M$-module categories.
\end{remark}

\subsection{Bimodule categories}

Let $\A,\, \B$ be  tensor categories. 

By definition, an {\em $(\A-\B)$-bimodule} category  $\M$ is
an $(\A \bt \B^{\opp})$-module category. 

Equivalently,  a category $\M$ is an $(\A-\B)$-bimodule category 
if it has left $\A$-module and right $\B$-module category structures compatible by a collection of isomorphisms
$a_{X,M,Y}:X*(M*Y)\to (X*M)*Y$  called {\em middle  associativity constraints}
natural in $X\in\A,\, Y\in\B,\,M\in \M,$ and such that the following diagrams 
$$\xygraph{ !{0;/r4.5pc/:;/u4.5pc/::}[]*+{X*(Y*(M*Z))} (
  :[u(.7)r(1.5)]*+{(X\otimes Y)*(M*Z)} ^{a_{X,Y,M*Z}}
  :[d(.7)r(1.5)]*+{((X\otimes Y)*M)*Z}="r" ^{a_{X\ot Y,M,Z}}
  ,
  :[r(.5)d(.8)]*+!R(.3){X*((Y*M)*Z))} _{1*a_{Y,M,Z}}
  :[r(2)]*+!L(.3){(X*(Y*M))*Z} _{a_{X,Y*M,Z}}
  : "r" _{a_{X,Y,M}*1}
)
}$$
$$\xygraph{ !{0;/r4.5pc/:;/u4.5pc/::}[]*+{X*(M*(Z\ot W))} (
  :[u(.7)r(1.5)]*+{(X*M)*(Z\ot W)} ^{a_{X,M,Z\ot W}}
  :[d(.7)r(1.5)]*+{((X*M)*Z)*W}="r" ^{a_{X*M,Z,W}}
  ,
  :[r(.5)d(.8)]*+!R(.3){X*((M*Z)*W)} _{1*a_{M,Z,W}}
  :[r(2)]*+!L(.3){(X*(M*Z))*W} _{a_{X,M*Z,W}}
  : "r" _{a_{X,M,Z}*1}
)
}$$
commute for all $X,\,Y\in\A,\, Z,\,W\in\B$, and $M\in\M$.

\begin{example}
A left  $\A$-module category $\M$ has a  
structure of an $(\A - (\A^*_\M)^{\opp})$-bimodule category.
\end{example}

\subsection{Tensor product of module categories and the Brauer-Picard group  of a tensor category }

Let $\A$ be a finite tensor category, let $\M$ be a right $\A$-module category, and let $\N$ be a left $\A$-module category.
The $\A$-module tensor product of $\M$ and $\N$  was defined in \cite[Section 3.1]{ENO}.  Let us recall this definition.
A bifunctor $F: \M \times \N \to \mathcal{K}$, where $\mathcal{K}$ is an abelian category  is called {\em $\A$-balanced} 
if there exists a family of isomorphisms $F(M\ot X,\, N) \xrightarrow{\simeq} F(M,\, X\ot N)$ natural in $M\in \M,\, N\in \N$,
and $X\in \A$ satsfying coherence axioms. Let $\Fun_{bal, re}(\M\times \N,\, \mathcal{K})$ denote
the category of $\A$-balanced functors from $\M \times   \N$   to $\mathcal{K}$ right exact in each variable. 
  
The $\A$-module tensor product of $\M$ and $\N$ is an Abelian category $\M \bt_\A \N$ together with 
$\C$-balanced bifunctor 
\[
B_{\M,\, \N}: \M \times \N \to \M \bt_\A \N
\]  
which is right exact in each variable and 
for every  abelian category $\mathcal{K}$ induces an equivalence 
\[
\Fun_{bal, re}(\M\times \N,\, \A) \simeq \Fun_{re}(\M \bt_\A \N,\, \mathcal{K}). 
\]
Here and below the subscript $re$ indicates that functors  under consideration are right exact. 
The existence of $\A$-module tensor product  was established  in  \cite[Section 3.2]{ENO}. Namely, it was shown that
\begin{equation}
\label{btA = Fun}
\M\bt_\A \N \simeq \Fun_{\A,re}(\M^{\opp},\, \N). 
\end{equation}
Note  that although  the categories considered in \cite{ENO} were assumed
to be semisimple the proof of this particular result does not use semisimplicity.  Indeed, first observe that $\M \bt \N$ is equivalent to
$\Fun_{re}(\M^{\opp},\, \N)$, since for $\M =\Rep(A)$ and $\N =\Rep(B)$, where $A$ and $B$ are algebras,   
both categories are identified with $\Rep(A\ot B)$. Next, by \cite[Proposition 3.5]{ENO} every balanced  
bifunctor $\M \times \N \to \mathcal{K}$ that is right  exact in every variable canonically factors through the functor
\[
\M \bt \N \simeq \Fun_{re}(\M^{\opp},\, \N) \xrightarrow{B_{\M,\N}}  \Fun_{\A,re}(\M^{\opp},\, \N),
\]
where ${B_{\M,\N}}$ is the left adjoint to the forgetful functor  
\[
\Fun_{\A,re}(\M^{\opp},\, \N) \to \Fun_{re}(\M^{\opp},\, \N).
\] 

Furthermore, if $\M$ and $\N$ are $\A$-bimodule categories then so is $\M \bt_\A \N$ (the $\A$-bimodule structure
on $\M \bt_\A \N$  is induced by the $\A$-bimodule structure on $\M \bt \N$). 

\begin{proposition}
Let $\M$ and $\N$ be  exact $\A$-bimodule categories. Then $\M \bt_\A \N$ is an exact $\A$-bimodule category.
\end{proposition}
\begin{proof}
It is enough to check that for all objects $F$ in $\M \bt_\A \N$  and projective objects  $P_1,\, P_2$ in $\C$
the object $P_1\ot F \ot P_2$ is projective. That is,  we need to show that the compositions of an
$\A$-module  functor $F: \M^{\opp} \to \N$ with the functors 
\begin{eqnarray*}
\M^{\opp} \to \M^{\opp} &:& M \mapsto M\ot P_1,\\
\N \to \N &:& N \mapsto N \ot P_2
\end{eqnarray*}
are projective 
objects in $\Fun_{\A}(\M^{\opp},\, \N)$.
 This is clear since the latter category is exact over $\A^*_\M$
and $\A^*_\N$  and the right multiplications by $P_1,\, P_2$ are $\A$-module endofunctors. 
\end{proof}

We say that an exact  $\A$-bimodule category $\M$ is {\em invertible} if there exists an exact  
$\A$-bimodule category $\N$ such that 
\[
\M \bt_\A \N \simeq \N \bt_\A \M \simeq \A,
\]
where $\A$ is viewed as an $\A$-bimodule category via the regular left and right actions of $\A$. 

\begin{remark}
\label{CM = C} 
It was proved in \cite[Propositon 4.2]{ENO} that
an $\A$-bimodule category $\M$ is invertible if and only if 
the tensor functor
\begin{equation}
\label{ENO criterion}
L: \A \to (\A^*_\M)^{\opp} : X \mapsto  ? \ot X
\end{equation}
is an equivalence.
\end{remark}

The group of equivalence classes of invertible $\A$-bimodule categories is called the {\em Brauer-Picard} group 
of $\A$ and is denoted by $\BrPic(\A)$.

\subsection{Module categories over braided tensor categories}
\label{modules over braided}

Let now $\C$ be a braided tensor category with the braiding 
\[
c_{X,Y} : X \ot Y \xrightarrow{\simeq} Y \ot X,\qquad X,\,Y \in \C.
\] 
The braiding of $\C$  gives  a tensor structure on the multiplication functor
$\C \bt \C \to \C$ \cite{JS}.  Hence, there is a canonical tensor functor
\begin{equation}
\label{mult is tensor}
\otimes: \C \bt \C^{\opp} \simeq \C \bt \C \to \C. 
\end{equation}

This allows to turn any left $\C$-module category $\M$ into a $\C$-bimodule category
as follows. The right action is $M* X := X * M$ for all $X\in \C$ and $M\in \M$.
Let $a_{X,Y,M}: X\ot(Y\ot M) \xrightarrow{\simeq} (X\ot Y)\ot M$ denote the 
left $\C$-module associativity constraint of $\M$.  The right
$\C$-module associativity constraint  of $\M$ is given by 
\begin{equation}
\label{right associativity}
\xymatrix{
(M*X)* Y \ar@{=}[d] \ar[rr]^{a_{M,X,Y}}  & & M *(X\ot Y) \ar@{=}[d] \\
Y *(X* M) \ar[r]^{a_{Y,X,M}} & (Y\ot X)*M  \ar[r]^{c_{Y,X}} & (X\ot Y) *M
}
\end{equation}
and the middle associativity constraint is given by
\begin{equation}
\label{middle associativity}
\xymatrix{
X*(M*Y)  \ar@{=}[d] \ar[rrr]^{a_{X,Y,M}} &&& (X*M)*Y \ar@{=}[d] \\
 X *(Y* M)  \ar[r]^{a_{X,Y,M}} & (X\ot Y)*M  \ar[r]^{c_{X,Y}}& 
(Y\ot X) *M   \ar[r]^{a_{Y,X,M}^{-1}} & Y * (X * M)
}
\end{equation}
for all $X,\,Y\in \C$ and $M\in \M$.  

Let $\mathbf{Mod}(\C)$ and $\mathbf{Bimod}(\C)$ denote the $2$-categories
of exact module and bimodule categories over $\C$, respectively. 
The above tensor functor \eqref{mult is tensor}  yields a $2$-functor 
\begin{equation}
\label{2functor B}
\B:  \mathbf{Mod}(\C) \to \mathbf{Bimod}(\C). 
\end{equation}
Clearly, the $2$-functor $\B$ is an embedding of $2$-categories.

\begin{definition}
We will call a $\C$-bimodule category  {\em one-sided} if it is equivalent to $\B(\M)$
for some left $\C$-module category $\M$.
\end{definition}

\begin{remark}
One can give an explicit characterization of one-sided categories. Namely,
a $\C$-bimodule category $\M$ is one-sided if it is equipped with a collection of isomorphisms 
\begin{equation}
\label{dXM}
d_{M,X}:M*X\to X*M,
\end{equation}
natural in $X\in\C$ and $M\in\M$, such that the following  diagrams commute:
\begin{equation}\label{coh1}\xymatrix{& M*(X\otimes Y) \ar[dl]_{a_{M,X,Y}} \ar[r]^{d_{M,X\ot Y}} & (X\otimes Y)*M \\
(M*X)*Y \ar[rd]_{d_{M,X}1} &&& X*(Y*M) \ar[ul]_{a_{X,Y,M}} \\
& (X*M)*Y \ar[r]^{a^{-1}_{X,M,Y}} & X*(M*Y) \ar[ur]_{1d_{M,Y}} }\end{equation}
\begin{equation}\label{coh2}\xymatrix{& (X*M)*Y\ar[r]^{d_{X*M,Y}} & Y*(X*M) \ar[dr]^{a_{Y,X,M}}  \\
X*(M*Y) \ar[rd]_{1*d_{M,Y}} \ar[ru]^{a_{X,M,Y}} &&& X*(M*Y), \\
& X*(Y*M) \ar[r]^{a_{X,Y,M}} & (X\ot Y)*M \ar[ur]_{c_{X,Y}*1} }\end{equation}
where $a$ denotes the associativity constraint of $\M$.
\end{remark}

Given left $\C$-module categories $\M$ and $\N$ there is an  obvious $\C$-bimodule equivalence
\[
\B(\B(\M) \bt_\C \N )\cong \B(\M)\bt_\C \B (\N). 
\]

Hence, when $\C$ is braided the group $\BrPic(\C)$ contains a subgroup $\Pic(\C)$ 
consisting of equivalence classes of one-sided invertible $\C$-bimodule categories. 
Following \cite{ENO} we call this group the {\em Picard group} of $\C$.

In what follows we will omit the $2$-functor $\B$ from notation and identify invertible 
$\C$-module categories with their images in  $\mathbf{Bimod}(\C)$.

\subsection{The $\alpha$-induction}
\label{alpha}

Let $\C$ be a braided tensor category  and let $\M$ be a $\C$-module category.
There is a pair of tensor functors (see \cite{BEK, O1}):
\begin{equation}
\label{def alpha ind}
\alpha_\M^\pm : \C \to \C^*_\M
\end{equation}
defined as follows.  For each $X\in \C$ the endofunctors $\alpha^\pm_\M(X) : \M \to \M$
coincide with the left multiplication by $X$, i.e., 
\[
\alpha^\pm_\M(X) = X \ot -. 
\]
Their  $\C$-module functor  structures are given by
\begin{eqnarray*}
\alpha^+_\M(X)(Y\ot M) &=& X\ot Y \ot M \xrightarrow{c_{X,Y}} Y \ot X \ot M = Y\ot \alpha^+_\M(X)(M)\qquad\mbox{and}\\
\alpha^-_\M(X)(M\ot Y) &=&  X\ot Y \ot M  \xrightarrow{c_{Y,X}^{-1}}  Y \ot X \ot M = Y \ot \alpha^-_\M(X)(M), 
\end{eqnarray*}
for all $X,\, Y\in \C$ and $M\in \M$. Here we suppress the associativity constraints.

When $\M$ is invertible the functors $\alpha^\pm_\M$ are equivalences and a functor $\partial_\M:\C\to\C$ defined by
\begin{equation} 
\label{def partial M}
(\alpha_\M^-)\circ \partial_\M = \alpha_\M^+
\end{equation}
is a braided autoequivalence of $\C$. The assignment $\M \mapsto \partial_\M$ gives
rise to a group homomorphism: 
\begin{equation}
\label{def partial}
\partial : \Pic(\C)\to \Aut^{br}(\C) \qquad \M \mapsto \partial_\M.
\end{equation}
To be precise the condition \eqref{def partial M} defines a tensor autoequivalence of $\C$. 
The reason why it is braided is explained in Remark~\ref{imp} (see also \cite{ENO} for details in the fusion case).

\section{The Picard crossed module of a braided tensor category}
\label{Picard section}

subsection{Algebras and their modules}

We refer the reader to \cite{O1} for basic definitions  and facts about algebras in tensor categories
and modules over them.

Let $A$ be an algebra in a tensor category $\A$ with the multiplication $\mu : A \ot A \to A$
and let $M$ be a right $A$-module in $\A$
with the structural map $\nu : M \ot A \to M$. 
For any  $X\in\A$ there is an $A$-module structure on  $X\otimes M$  defined by 
\[
\id_X\otimes\nu:X\otimes M\otimes A \to X\otimes M.
\]  
Thus the category $\A_A$ of right $A$-modules in $\A$ 
is a left $\A$-module category via
$$
\A\times\A_A\to\A_A,\quad (X,\,M)\mapsto X\otimes M.
$$
Similarly, the category  ${}_A\A$ of left  $A$-modules in  $\A$ is a right $\A$-module category.

\begin{remark}
\label{left right op}
Let $A$ be an algebra in $\A$. Then the left $\A$-module category $({}_A\A)^\opp$ 
is equivalent to  $\A_A$. 
\end{remark}

It was shown in \cite{EO} that every  left (respectively, right)
$\A$-module category is equivalent to $\A_A$ 
(respectively, to ${}_A\A$) for some algebra $A$ in $\A$.

Let $A$ be an algebra in a tensor category $\A$ and $\M$ be a left $\A$-module category. 
Define ${_A}\M$ (the category of {\em $A$-modules in $\M$}) as the category of pairs $(M,m)$, 
where $M$ is an object of $\M$ and $m:A*M\to M$ is a morphism in $\M$ such that the diagram
$$\xymatrix{A*(A*M) \ar[dd]_{a_{A,A,M}} \ar[rr]^{1*m} && A*M \ar[rd]^m \\ & & & M\\ (A\ot A)*M \ar[rr]^{\mu*1} && A*M \ar[ru]_m}$$
commutes. 

A morphism between $(M,\, m)$ and $(M',\, m')$ is a morphism $f: M\to M'$ such that
$f\circ m =m' \circ (\id_A * f)$.

\begin{lemma}
\label{upmc}
Let $\A$ be a finite tensor category and let $\M$ be an exact right $\A$-module category.
The functor 
\begin{equation}
\label{Fun to AM}
T: \Fun_\A(\A_A,\M) \to {_A}\M :  F\mapsto F(A)
\end{equation}
is an equivalence of 
categories. 
\end{lemma}
\begin{proof}
For any $\A$-module functor $F:\A_A\to\M$ the object $F(A)\in\M$ has a structure of an $A$-module:
\begin{equation}
\label{mofu}
A*F(A) \xrightarrow{\simeq}   F(A\ot A) \xrightarrow{F(\mu)} F(A),
\end{equation}
where the first arrow  is given by the $\A$-module structure of $F$ and the second arrow is the image
of the multiplication of $A$. It is easy to see that  $\A$-module transformations between $\A$-module functors $F,\,G$ 
correspond to morphisms of $A$-modules $F(A),\, G(A)$ in $\M$.  Thus, $T$ is a well defined functor. 

Define a functor $S: {_A}\M  \to \Fun_\A(\A_A,\M)$ by
$M \mapsto S_M$, where $S_M(X) = X \ot_A M$. It is clear that $S_M$ is an $\A$-module functor
and that $T\circ S$ is isomorphic to the identity endofunctor of ${_A}\M$. 

Also, $S \circ T$ is isomorphic to the identity functor since
for every $\A$-module functor $F: \A_A\to \M$ and a right  $A$-module $X$ in $\A$
there is a natural isomorphism  $X\ot_A F(A) \cong F(X)$. 
Thus, $T$ is an equivalence.
\end{proof}

A particular case of Lemma~\ref{upmc}
that will be useful for us later is the category of  $A$-modules in $\M=\A_B$,
where $B$ is an exact algebra in $\A$. The category ${_A}\A_B$ is  the category of {\em ($A$-$B$)-bimodules in $\C$}. 
\begin{corollary}
\label{fmc}
The functor
$$\Fun_\A(\A_A,\A_B) \to {_A}\A_B,\quad F\mapsto F(A)$$
is an equivalence of categories. 
\end{corollary}

\subsection{Tensor product of algebras in a braided category}
\label{tpa}

Let now $\C$ be a braided tensor category and let $A$ be an algebra in $\C$. 
Given a left $\C$-module category $\M$,
the braiding in $\C$ allows us to turn ${_A}\M$ into a left $\C$-module category.
In this situation the functor $\Fun_\C(\C_A,\M) \xrightarrow{\sim} {_A}\M$ from Lemma~\ref{upmc}
is an equivalence of  $\C$-module categories.

It is well known that for braided $\C$ the tensor product $A\ot B$ of two algebras $A,B\in\C$ has an algebra structure, 
with the multiplication map $\mu_{A\ot B}$ defined as 
$$
A\ot B\ot A\ot B \xrightarrow{\id_A\ot  c_{B,A}\ot \id_B }  A\ot A\ot B\ot B \xrightarrow{\mu_A\ot \mu_B}  A\ot B,
$$ 
where $\mu_A$ and $\mu_B$ are multiplications of algebras $A$ and $B$, respectively (here we
suppress the associativity constraints in $\C$).

Let $A^{\op}=A$ denote the algebra with the multiplication  opposite to that of $A$:
\[
A\ot A \xrightarrow{c_{A,A}} A\ot A \xrightarrow{\mu_A} A.
\]

\begin{proposition}
\label{AB}
Let $\C$ be a braided tensor category and  let $A$ and $B$ be exact algebras in $\C$.
Then $$\C_A \bt_\C \C_B \simeq \C_{A\ot B}$$ as $\C$-module categories.
\end{proposition}
\begin{proof}
Note that a left $\C$-module category $\C_A$ considered as a right $\C$-module category is equivalent to ${}_{A^\opp}\C$. 
By Remark~\ref{left right op}
the opposite category $({}_{A^\opp}\C)^{\opp}$ is equivalent to $\C_{A^{\op}}$ as a left $\C$-module category. 

Hence, using \eqref{btA = Fun} and Corollary~\ref{fmc} we obtain
\[
\C_A \bt_\C \C_B \simeq \Fun_{\C}(({}_{A^\opp}\C)^{\op},\, \C_B) 
\simeq \Fun_{\C}(\C_{A^{\op}},\, \C_B) \simeq {}_{A^{\op}}\C_B \simeq \C_{A\ot B}, 
\]
since an $(A\ot B)$-module in $\C$  is the same thing as an $(A^{\opp}-B)$-bimodule. 
\end{proof}

\subsection{Azumaya algebras}
\label{azumaya section}

Here we recall the characterization of algebras in $\C$ whose categories of modules are invertible.

Let $A$ be an exact algebra in a braided tensor category $\C$. 

Note that the multiplication on $A$ 
$$
A\ot A^{\opp}\ot A \xrightarrow{\id_A \ot c_{A,A}}   A\ot A\ot A \xrightarrow{\mu_A\ot \id_A} A\ot A \xrightarrow{\mu_A} A
$$
induces a homomorphism of algebras
\begin{equation}
\label{az}
A\ot A^{\opp}\to A\ot A^*,
\end{equation}
where $A^*$ is the dual object to $A$ and the multiplication in $A\ot A^*$ is defined using the evaluation morphism. 

\begin{definition}
An algebra $A$ in a braided tensor category $\C$ is {\em Azumaya} if the map \eqref{az} is an isomorphism.
\end{definition}

It was established in   \cite[Theorem 3.1]{OZ} that
$A$ is an Azumaya algebra  if and only if the tensor functors  
\[
\alpha^\pm_{\C_A}:\C\to {_A}\C_A
\] 
defined in \eqref{def alpha ind}  are equivalences. Thus,  the Picard group of $\C$ is isomorphic
to the group of Morita equivalence classes of exact Azumaya algebras (the latter group was considered in
\cite{OZ}). 

Let $A$ be an exact Azumaya algebra in $\C$.  
Let $\partial_A = \partial_{\C_A}$  denote the braided  autoequivalence 
introduced  in \eqref{def partial}.  By definition of $\partial_A$, there exists 
a  natural isomorphism of right $A$-modules
\[
\phi_X:  A\ot X \xrightarrow{\simeq} \partial_A(X) \ot A, \quad X\in \C.
\]
This means that  the following diagram commutes:
\begin{equation}
\label{phi is bim}
\xymatrix{
A\ot X \ot A \ar[rrr]^{\phi_X\ot \id_A}  \ar[d]_{c_{A,X} \ot \id_A} &&&  
\partial_A(X) \ot A  \ot A\ar[ddd]^{\id_{\partial_A(X)}\ot \mu_A} \\
X\ot A \ot A \ar[d]_{\id_X\ot \mu_A} &&&  \\
X\ot A \ar[d]_{c_{X,A}} &&& \\
A\ot X \ar[rrr]^{\phi_X}  &&& \partial_A(X) \ot A. 
}
\end{equation}

The tensor structure
\[
\nu_{X, Y} :     \partial_A(X\ot Y) \xrightarrow{\simeq}  \partial_A(X)  \ot \partial_A(Y)  ,\qquad X,Y\in \C, 
\]
of $\partial_A$ satisfies the following commutative diagram:
\begin{equation}
\label{tensoriality of phi+}
\xymatrix{
A\ot X \ot Y \ar[rrr]^{\phi_X\ot \id_Y} \ar[d]_{\phi_{X\ot Y}} &&&
\partial_A(X) \ot A \ot Y \ar[d]^{\id_{\partial_A(X)} \ot \phi_Y} \\
\partial_A(X\ot Y) \ot A \ar[rrr]^{\nu_{X,Y}} &&& \partial_A(X) \ot \partial_A(Y) \ot A.
}
\end{equation}

\begin{lemma}
\label{tensoriality of phi++}
The following diagram
\begin{equation}
\label{tensoriality of phi}
\xymatrix{
A\ot X\ot A \ot Y \ar[rr]^{\phi_X \ot \phi_Y} \ar[d]_{ c_{X,A} } &&
\partial_A(X)\ot A \ot \partial_A(Y)\ot A \ar[d]^{c_{A,\partial_A(Y)}}  \\
A\ot A\ot X\ot Y \ar[d]_{ m_A} &&
\partial_A(X) \ot \partial_A(Y)\ot A \ot A\ar[d]^{{ m_A}} \\
A\ot X\ot Y \ar[r]^{\phi_{X\ot Y}} & \partial_A(X\ot Y) \ot A \ar[r]^{\nu_{X,Y}}  &
\partial_A(X) \ot \partial_A(Y) \ot A.
}
\end{equation}
is commutative (here, as usual, we suppress the associativity constraints and identity morphisms).
\end{lemma}
\begin{proof}
Note that compositions of the left and the right vertical arrows in diagram \eqref{tensoriality of phi}
coincide, respectively, with the 
canonical epimorphisms
\begin{eqnarray*}
A\ot X\ot A \ot Y  &\to& (A\ot X)\ot_A (A \ot Y) \cong A \ot X \ot Y \quad \mbox{and} \\
\partial_\A(X)\ot A \ot \partial_\A(Y)\ot A &\to& (\partial_A(X)\ot A) \ot_A (\partial_A(Y)\ot A) \\
&   & \cong \partial_\A(X) \ot \partial_\A(Y)\ot A .
\end{eqnarray*}
Hence, the following diagram
\begin{equation}
\label{tensoriality of phi'}
\xymatrix{
A\ot X\ot A \ot Y \ar[rr]^{\phi_X \ot \phi_Y} \ar[d]_{ c_{X,A} } &&
\partial_A(X)\ot A \ot \partial_A(Y)\ot A \ar[d]^{c_{A,\partial_A(Y)}}  \\
A\ot A\ot X\ot Y \ar[d]_{\mu_A} &&
\partial_A(X) \ot \partial_A(Y)\ot A \ot A\ar[d]^{{ m_A}} \\
A\ot X\ot Y \ar[r]^{\phi_X} & \partial_A(X) \ot A\ot Y  \ar[r]^{\phi_Y}  &
\partial_A(X) \ot \partial_A(Y) \ot A.
}
\end{equation}
is commutative by functoriality  of $\ot_A$.  But the bottom row composition in diagram \eqref{tensoriality of phi'}
coincides with that of diagram \eqref{tensoriality of phi} by the identity \eqref{tensoriality of phi+}.
\end{proof}

Let $B$ be an algebra in $\C$ and suppose that $A$ is an Azumaya algebra in $\C$.
Then $\partial_A(B)$ is also an algebra in $\C$.  We will denote by
\[
\mu_B: B \ot B \to B\quad  \mbox{ and  } \quad
\mu_{\partial_A(B)}: \partial_A(B) \ot \partial_A(B) \to \partial_A(B)
\]
the multiplications of $B$ and $\partial_A(B)$ respectively.

\begin{proposition}
\label{AB=BA}
The morphism $\phi_B: A\ot B\to \partial_A(B)\ot A$ is an isomorphism of algebras.
\end{proposition}
\begin{proof}
Consider the following diagram:
\begin{equation}
\label{phi is iso}
\xymatrix{
A\ot B\ot A \ot B \ar[rr]^{\phi_B \ot \phi_B} \ar[d]_{ c_{B,A} } &&
\partial_A(B)\ot A \ot \partial_A(B)\ot A \ar[d]^{c_{A,\partial_A(B)}}  \\
A\ot A\ot B\ot B \ar[d]_{ \mu_A} &&
\partial_A(B) \ot \partial_A(B)\ot A \ot A\ar[d]^{{ \mu_A}} \\
A\ot B\ot B \ar[r]^{\phi_{B\ot B}} \ar[d]_{\mu_B} & \partial_A(B\ot B) \ot A \ar[r]^{\nu_{B,B}}  &
\partial_A(B) \ot \partial_A(B) \ot A \ar[d]^{\mu_{\partial_A(B)}} \\
A\ot B  \ar[rr]^{\phi_B} && \partial(B) \ot A.
}
\end{equation}
The upper subdiagram is commutative by Lemma~\ref{tensoriality of phi++} 
and the lower subdiagram is the definition of multiplication  $\mu_{\partial_A(B)}$. Hence, diagram \eqref{phi is iso} is commutative. 
This is precisely the property of $\phi_B$ being an algebra homomorphism.
\end{proof}

\subsection{Definition of the Picard crossed module}
\label{def Pic cr mod}

\begin{definition}     
\label{2Whitehead}
A {\em crossed module} $(G,\, C)$ is a pair of groups $G$ and $C$
together with an action of $G$ on $C$, denoted $(g,\,c)\mapsto {}^g c$,
and a homomorphism $\partial: C\to G$ satisfying
\begin{eqnarray} 
\label{crossed module1}
\partial({}^g c) &=& g \partial(c) g^{-1}, \\
\label{crossed module2}
{}^{\partial(c)} c' &=& cc'c^{-1}\quad  c,c'\in C, g\in G.
\end{eqnarray}
Let  $(G_1,\, C_1)$ and $(G_2,\, C_2)$ be  crossed modules with structural
maps $\partial_1: C_1\to G_1$ and  $\partial_1: C_2\to G_2$. A {\em homomorphism}
between these crossed modules  is a pair of group homomorphisms $\gamma: G_1\to G_2$ 
and $\varphi: C_1\to \C_2$ such that  $\partial_2\circ \phi =\gamma \circ \partial_1$ and 
$\phi({}^g c) = {}^{\gamma(g)} \phi(c)$ for all $c\in C_1$ and $g\in G_1$
\end{definition}

\begin{remark}
\label{Kernel d}
It is clear that the kernel of the homomorphism $\partial$ in Definition~\ref{2Whitehead}
is a subgroup of the center of $C$ and the image of $\partial$  is a normal subgroup of $G$.
\end{remark}

Let $\C$ be a braided tensor category. Set
\begin{equation}
\label{shorthand}
G :=\Aut^{br}(\C),\quad C:= \Pic(\C).
\end{equation}
In \eqref{def partial} we defined a canonical homomorphism 
\[
\partial: C \to G :\M \mapsto \partial_\M. 
\]

There is also a canonical action of $\Aut^{br}(\C)$ on $\Pic(\C)$. Namely, for $g\in \Aut^{br}(\C)$ and  a $\C$-module 
category $\M$  the category ${}^g \M$ is defined as follows. As an abelian category, ${}^g \M = \M$.
The action of $\C$ on $\M$ is defined by
\[
X\odot M := g^{-1}(X)*M ,\qquad \mbox{ for all } M\in \M,\, X\in \C.
\]
Note that for an algebra $A\in\C$ the $\C$-module category ${}^g (\C_A)$ is equivalent to $\C_{g(A)}$. Here $g(A)$ is the algebra with multiplication $\mu_{g(A)} = g(\mu_A)$. 

\begin{theorem}
\label{crmodthm}
The pair $(G,\,C)= (\Aut^{br}(\C),\,\Pic(\C))$  equipped with the above structural operations  is a crossed module.
\end{theorem}
\begin{proof}
To check the axiom \eqref{crossed module1}, 
note that tensor equivalences  
\[
\alpha^\pm_{{}^g\M}: \C \to  \C^*_{{}^g\M}
\]
defined in \eqref{def alpha ind} satisfy $\alpha^\pm_{{}^g\M}\cong \alpha^\pm_\M \circ g^{-1}$. Hence,
\begin{equation}
\lb{fpcm}
\partial_{{}^g\M}\cong (\alpha^-_{{}^g\M})^{-1} \circ \alpha^+_{{}^g\M}\cong g\circ \partial_\M \circ g^{-1},
\quad \mbox{for all } \M \in \Pic(\C),\, g\in G.
\end{equation}

Let us check axiom~\eqref{crossed module2}.  
Take $\M,\,\N\in \Pic(\C)$  and let $A$ and $B$ be algebras in $\C$
such that $\M\simeq \C_A$ and $\N\simeq \C_B$.  By Proposition~\ref{AB} we have 
\[
\M \bt_\C \N \simeq \C_{A\ot B} \quad \mbox{ and } \quad
{}^{\partial_\M}\N \bt_\C \M \simeq \C_{\partial_\M(B) \ot A}.
\]
Since by Proposition~\ref{AB=BA} the algebras $A\ot B$ and $\partial_\M(B) \ot A$ are isomorphic,
we conclude  $\M \bt_\C \N  \simeq {}^{\partial_\M}\N \bt_\C \M$, as required.
\end{proof}

\begin{definition}
We will call the pair $(\Aut^{br}(\C),\,\Pic(\C))$ the {\em Picard crossed module} of $\C$ and denote it 
$\mathfrak{P}(\C)$.
\end{definition}

\section{Picard crossed module and braided autoequivalences of the center}
\label{section braided aut}

In this Section we give a characterization of the Picard crossed
module  of a braided tensor category $\C$ in terms of  braided autoequivalences of $\Z(\C)$.

\subsection{The Brauer-Picard group and braided autoequivalences of the center}
\label{Here Phi is defined}

Let $\M$ be an exact left $\C$-module category.  It can be regarded as a $(\C \bt \C^*_\M)$-module category.
The following constructions are taken from \cite[Section 3.4]{EO}.  There are canonical equivalences
\begin{equation}
\label{ZC = ZCM}
a_\M : \Z(\C) \xrightarrow{\sim} (\C \bt \C^*_\M)^*_\M : (Z,\,\gamma) \mapsto Z * ?, 
\end{equation}
where the left $\C$-module functor structure of $a_\M(Z,\gamma)$ is given by
\begin{equation}
\label{aM left}
X*(Z*M) \xrightarrow{a_{X,Z,M}} (X\ot Z)*M \xrightarrow{\gamma_X} (Z\ot X)*M
\xrightarrow{a_{Z,X,M}^{-1}}  Z*(X*M),
\end{equation}
for all $X\in \C$ and $M\in \M$, and its left $\C^*_\M$-module functor structure 
\begin{equation}
\label{aM  right}
F(Z*M) \xrightarrow{\simeq}  Z* F(M),
\end{equation}
for $F\in \C^*_\M$ is given using the $\C$-module functor structure of $F$.

One defines a functor 
\begin{equation}
\label{tilde aM}
\tilde{a}_\M :   \Z(\C^*_\M ) \xrightarrow{\sim} (\C \bt \C^*_\M)^*_\M 
\end{equation}
in an analogous way. 

The composition $\tilde{a}_\M^{-1}\circ a_\M$ is a braided  tensor equivalence between
$\Z(\C)$ and $ \Z(\C^*_\M)^\rev = \Z((\C^*_\M)^{\opp})$.

When $\M$ is an invertible $\C$-bimodule category the composition of $\tilde{a}_\M$ and the braided tensor 
equivalence   $\Z(\C)\xrightarrow{\sim} \Z((\C^*_\M)^{\opp})$  induced by the tensor equivalence
\[
L: \C \xrightarrow{\sim} (\C^*_\M)^{\opp} : X \mapsto ? * X
\]
from Remark~\ref{CM = C} gives a tensor equivalence
\begin{equation}
\label{bM}
b_\M : \Z(\C) \xrightarrow{\sim} (\C \bt \C^*_\M)^*_\M : (Z,\,\gamma) \mapsto ? * Z
\end{equation}
where  the left $\C$-module functor structure of $b_\M(Z,\, \gamma)$ is given by 
the middle associativity constraint of $\M$:
\begin{equation}
\label{bM left}
X*(M*Z)  \xrightarrow{a_{X,M,Z}} (X*M)*Z,
\end{equation}
while the right $\C$-module functor structure
(which is the same as the  left $C^*_\M$-module functor structure
upon the identification $\C^*_\M \simeq \C^{\opp}$)
of $b_\M(Z,\, \gamma)$ is given using the right $\C$-module associativity constraint of $\M$
and the half braiding:
\begin{equation}
\label{bM right}
(M*Z)*Y \xrightarrow{a_{M,Z,Y}} M*(Z\ot Y) \xrightarrow{\gamma_Y^{-1}} M*(Y\ot Z)
\xrightarrow{a_{M,Y,Z}^{-1}} (M*Y)*Z,
\end{equation}
for all $X,\,Y\in \C$ and $M\in \M$.

Thus, we have a canonical braided tensor autoequivalence
\begin{equation}
\label{PhiM}
\Phi(\M) = b_\M^{-1}\circ a_\M :\Z(\C)\to \Z(\C).
\end{equation}

The following result was proved  in \cite[Section 5]{ENO} when $\C$ is a fusion category.  This 
argument carries over verbatim to the case of finite tensor categories. We recall  the proof
for the reader's convenience and also for the future references. 

\begin{theorem}
\label{BrPic=Autbr theorem}
Let $\C$ be a finite tensor category. The assignment $\M \mapsto \Phi(\M)$, where
$\Phi(\M)$ is defined in \eqref{PhiM} gives rise to a group isomorphism 
\begin{equation}
\label{big thing}
\Phi : \BrPic(\C) \xrightarrow{\simeq} \Aut^{br}(\Z(\C)).
\end{equation}
\end{theorem}
\begin{proof}
To see that $\Phi$ is a homomorphism  observe that the  $\C$-bimodule 
functor of  right multiplication by an object $Z\in \Z(\C)$ on  $\M \bt_\C \N$, where
$\M$ and $\N$ are invertible $\C$-bimodule categories, 
is isomorphic  to the well-defined functor of ``middle" multiplication by $\left(\Phi(\N)\right)(Z)$,
which, in turn,  is isomorphic to the functor of left multiplication by $\left(\Phi(\M)\circ \Phi(\N)\right)(Z)$.
This gives a natural isomorphism of tensor functors 
$\Phi(\M) \circ\Phi(\N) \cong \Phi(\M\bt_\C \N)$.  Hence, $\Phi$ is a homomorphism.

Let us recall the construction of the map
\begin{equation}
\label{Psi}
\Psi : \Aut^{br}(\Z(\C)) \to \BrPic(\C).
\end{equation}
inverse to the homomorphism \eqref{big thing}. 

Let $F: \Z(\C)\to \C$ and $I:\C\to \Z(\C)$ denote the canonical forgetful functor and its right adjoint.
Given a braided autoequivalence $\alpha \in \Aut^{br}(\Z(\C))$ let $L_\alpha:=
\alpha^{-1}(I(\be))$. 
The category  ${}_{L_\alpha}\Z(\C)$ is a finite tensor category with respect to $\otimes_{L_\alpha}$.

Let us show that the algebra $F(L_\alpha) \in \C$ is exact, i.e., that the category ${}_{L_\alpha}\C$
of  $F(L_\alpha)$-modules in $\C$ is exact. By Lemma~\ref{upmc}  this category is equivalent
to $\Fun_{\Z(\C)}(\Z(\C)_{L_\alpha},\, \C)$ as a $\C$-module category. By Remark~\ref{great EO}
the latter category is exact as a $\Fun_{\Z(\C)}(\C,\, \C)$-module category. In particular, it is exact
as a $\C$-module category.

Let
\[
F(L_\alpha) = \bigoplus_{i\in J}\, L_\alpha^i
\]
be the decomposition of $F(L_\alpha)$ into a direct sum of indecomposable exact algebras in $\C$.  

For any $i\in J$ 
the composition
\begin{equation}
\label{crucial equiv}
\C \xrightarrow{\iota} {}_{L_\alpha} \Z(\C)  \xrightarrow{F}  {_{F(L_\alpha)} \C_{F(L_\alpha)}} 
\xrightarrow{\pi_i} {_{L_\alpha^i}\C_{L_\alpha^i}}
\end{equation}
is a tensor equaivalence, where 
\begin{equation}
\label{iota}
\iota : \C \xrightarrow{\sim} {}_{L_\alpha} \Z(\C)  :  X \mapsto  \alpha^{-1}(I(X))
\end{equation}
and $\pi_i$ is a projection from $ {}_{F(L_\alpha)} \C_{F(L_\alpha)}  = \oplus_{i,j\in J}\, {}_{L_i} \C_{L_j} $ to the  $(i,\,i)$ component.

Hence, $\C_{L_i}$ gets a structure of an invertible $\C$-bimodule category.  Its equivalence
class  does not depend on a particular $i\in J$. One sets $\Psi(\alpha) := \C_{L_i}$.

The verification of identities $\Phi \circ \Psi =\id$ and $\Psi\circ \Phi =\id$ is the same as 
in \cite[Section 5.3]{ENO}.
\end{proof}

\begin{remark}
Note that $\BrPic(\C)$ and $\Aut^{br}(\Z(\C))$ are monoidal groupoids (i.e., monoidal categories
in which every object is invertible)
In fact, the assignment \eqref{big thing} is a monoidal equivalence rather than just a group isomorphism,
see \cite[Section 5]{ENO}.
\end{remark}

\subsection{The image of $\Pic(\C)$ in $\Aut^{br}(\Z(\C))$}
\label{Picard-braided}

Recall from Section~\ref{modules over braided} that  the group $\BrPic(\C)$
contains a subgroup $\Pic(\C)$ consisting of equivalence classes of invertible
$\C$-module categories (regarded as one-sided $\C$-bimodule categories).

Our goal now is to describe the image of $\Pic(\C)$ in $\Aut^{br}(\Z(\C))$ under isomorphism~\eqref{big thing}.

Let $\Aut^{br}(\Z(\C); \C) \subset \Aut^{br}(\Z(\C))$  denote the subgroup consisting of 
isomorphisms classes of  braided autoequivalences of $\Z(\C)$ trivializable on $\C$,
see Definition~\ref{trivializable def}. 

The  next Theorem was suggested to us by V.~Drinfeld. 

\begin{theorem}
\label{Drinfeld's Theorem}
Let $\C$ be a braided tensor category.
The canonical isomorphism $\Phi: \BrPic(\C) \xrightarrow{\simeq} \Aut^{br}(\Z(\C))$ restricts to an isomorphism
\begin{equation}\label{driso}
\Phi |_{\Pic(\C)} :  \Pic(\C) \xrightarrow{\simeq} \Aut^{br}(\Z(\C); \C).
\end{equation}
\end{theorem}
\begin{proof}
First, let us show that $\Phi(\Pic(\C)) \subset \Aut^{br}(\Z(\C); \C)$.
Let $\M$ be an invertible one-sided $\C$-module category.
Let  
$\Phi(\M) \in \Aut^{br}(\Z(\C))$
be the braided autoequivalence  of $\Z(\C)$ defined in Section~\ref{Here Phi is defined}.
The equivalences $a_\M$ and $b_\M$ defined in \eqref{ZC = ZCM} and \eqref{bM}
can be explicitly described as follows.  
Let $(Z,\, \gamma)$ be an object in $\Z(\C)$, where 
\[
\gamma_X: X\ot  Z \to Z\ot X,\quad X\in \C,
\]
is the half braiding.
Then $a_\M(Z,\, \gamma)(M) = Z* M$ 
and its left and right $\C$-module functor structures are found by translating
\eqref{aM left} and \eqref{aM right} to our setting: 
\begin{equation}
\label{aM computation}
\xymatrix{
X * (Z * M)  \ar[d]^{a_{X,Z,M}} \ar[rrr]^{\sim}  &&&  Z * (X * M) \\
 (X\ot Z) * M  \ar[rrr]^{\gamma_X} &&& (Z\ot X) * M \ar[u]_{a_{Z,X, M}^{-1}}
}
\end{equation}
and 
\begin{equation}
\label{aM computation1} 
\xymatrix{
(Z * M) * Y  \ar@{=}[d]    \ar[rrr]^{\sim} &&& Z* (M * Y)  \ar@{=}[d]    \\
Y* (Z*M) \ar[r]^{a_{Y, Z, M}} & (Y \ot Z) *M \ar[r]^{c_{Z,Y}^{-1}} & (Z \ot Y) *M \ar[r]^{a_{Z,Y,M}^{-1}} & Z*(Y*M),
}
\end{equation}
for all $X,\,Y \in \C$ and $M\in \M$, where $a$ denotes the left $\C$-module associativity constraint of $\M$. 

Also, $b_\M(Z,\, \gamma)(M) = M * Z  =  Z * M$  as a functor and its
left and right $\C$-module functor structures are found from \eqref{bM left} and \eqref{bM right}:
\begin{equation}
\label{bM computation}
\xymatrix{ 
X * (M *  Z)   \ar@{=}[d]    \ar[rrr]^{\sim} &&& (X * M ) * Z  \ar@{=}[d]    \\
X*(Z* M) \ar[r]^{ a_{X, Z, M}} &  (X\ot Z)*M \ar[r]^{c_{X,Z}} & (Z\ot X)*M \ar[r]^{a_{Z,X,M}^{-1}} & Z*(X*M)
}
\end{equation}
and
\begin{equation}
\label{bM computation1}
\xymatrix{
(M*Z)*Y  \ar@{=}[d]    \ar[rrr]^{\sim} &&& (M * Y)*Z \ar@{=}[d]  \\
Y*(Z*M) \ar[d]^{a_{Y, Z,M}} &&&  Z*(Y*M)          \\
(Y\ot Z)*M  \ar[r]^{c_{Y,Z}} & (Z\ot Y)*M  \ar[r]^{\gamma_Y^{-1}} & (Y\ot Z)*M  \ar[r]^{c_{Z,Y}^{-1}} & (Z\ot Y)*M
\ar[u]_{a_{Z,Y,M}^{-1}}
}
\end{equation}
for all $X,\,Y \in \C$ and $M\in \M$. 

The diagrams \eqref{bM computation} and \eqref{aM computation1}  are nothing but middle associativity
isomorphism \eqref{middle associativity} and its inverse.  The diagram \eqref{bM computation1} uses the right
$\C$-module associativity \eqref{right associativity}  and its inverse as well as the half braiding of $Z$. 

Since $\C$ is embedded into $\Z(\C)$ via 
\[
Z \mapsto (Z,\, c_{-,Z}),
\]
i.e., $\gamma_X = c_{X,Z}$ in this case, 
we see from  \eqref{aM computation}, \eqref{aM computation1} and \eqref{bM computation}, 
\eqref{bM computation1} that
the restrictions of  $a_\M$ and $b_\M$ on the subcategory $\C \subset \Z(\C)$ coincide, i.e.,
$\Phi(\M)$ is trivializable on $\C$. So  $\Phi(\Pic(\C)) \subset \Aut^{br}(\Z(\C); \C)$.

It remains to show that $\Phi(\Pic(\C)) = \Aut^{br}(\Z(\C); \C)$.
Let  $\alpha \in  \Aut^{br}(\Z(\C); \C)$. 
We need to  show that the equivalence class of invertible $\C$-bimodule 
category $\M:=\Psi(\alpha)$  (where $\Psi:  \BrPic(\C) \to \Aut^{br}(\Z(\C))$
is the inverse of $\Phi$, see \eqref{Psi}) is in $\Pic(\C)$.  

According to the description from the proof of Theorem~\ref{BrPic=Autbr theorem} 
$\M$ is equivalent  to any indecomposable component of the 
$\C$-module category $\C_{F(L_\alpha)}$  of left modules over the algebra
$F(L_\alpha)$, where   $ L_\alpha=  \alpha^{-1}(I(\be))\in \Z(\C)$.  
Thus, it suffices to show that the $\C$-bimodule category $\C_{F(L_\alpha)}$ is one-sided.

The left action of $X\in \C$ on $\C_{F(L_\alpha)}$ is via tensor multiplication: 
\begin{equation}
\label{M*X}
X*M  = X\ot M.
\end{equation}
The right action of $X$ is via module multiplication over $F(L_\alpha)$ 
with the image of $X$ under equivalence \eqref{crucial equiv}. Let us describe
this action explicitly. 
Since $I(X) \cong X\ot I(1)$ for all $X \in \C \subset \Z(\C)$ and $\alpha$ is trivializable on $\C$
we see that  equivalence \eqref{iota} in our situation becomes
\begin{equation}
\label{iota new}
\C \xrightarrow{\sim} {\Z(\C)}_{L_\alpha} : X \mapsto  X\ot L_\alpha.
\end{equation}
Therefore, the right action of $X$ on $\C_{F(L_\alpha)}$ is given by
\begin{equation}
\label{X*M}
M*X = M \ot_{F(L_\alpha)} (X\ot F(L_\alpha)) \cong  M \ot X
\end{equation}
for all $ X\in \C,\, M \in \C_{F(L_\alpha)}$. The action of  $F(L_\alpha)$ 
on $M*X  \cong M\ot X$ is given by
\[
M \ot X \ot F(L_\alpha) 
\xrightarrow{1\ot c_{X,F(L_\alpha)}}
M \ot F(L_\alpha)  \ot X \xrightarrow{\rho_M\ot 1} M \ot X, 
\]
where we omit the associativity constraints. Here
$\rho_M:M\ot F(L_\alpha)\to M$ denotes the $F(L_\alpha)$-module structure 
on $M$.

We have a natural family of $F(L_\alpha)$-module isomorphisms 
\[
d_{M, X} := c_{M, X} : M \ot X \to X \ot M. 
\]
To show that the $\C$-bimodule category $\C_{F(L_\alpha)}$
is one sided we  need to check that isomorphisms $d_{X, M}$
satisfy commutative diagrams \eqref{coh1} and \eqref{coh2}.
But these diagrams are nothing but hexagon axioms of the braiding.

Thus, $\Aut^{br}(\Z(\C); \C) \subset \Phi(\Pic(\C))$ and the proof
is complete.
\end{proof}

\subsection{A characterization of the Picard crossed module}

Let $\C$ be a finite braided tensor category. There is a canonical homomorphism
\begin{equation}
\label{Sigma}
\Sigma: \Aut^{br}(\Z(\C); \C)  \to \Aut^{br}(\C)
\end{equation}
defined as follows. 
Every braided autoequivalence $\alpha \in \Aut^{br}(\Z(\C))$ trivializable on $\C$ maps  the centralizer
$\C$ in $\Z(\C)$ to itself. This centralizer is $\C^\rev \subset \Z(\C)$.  Hence, $\alpha$ restricts
to a braided autoequivalence  of $\C^\rev$, i.e., to an element of $\Aut^{br}(\C^\rev) = \Aut^{br}(\C)$
which we denote $\Sigma(\alpha)$. 

\begin{lemma}
\label{two inductions}
Let $\C$ be a braided tensor  category. The composition
\[
\Pic(\C) \xrightarrow{\Phi}  \Aut^{br}(\Z(\C); \C)  \xrightarrow{\Sigma} \Aut^{br}(\C)
\]
coincides with homomorphism $\partial :  \Pic(\C) \to \Aut^{br}(\C)$ defined in \eqref{def partial}.
\end{lemma}
\begin{proof}
We need to show that for each invertible
$\C$-module category $\M$  the restriction of the braided autoequivalence $\Phi(\M)$
on $\C^\rev \subset \Z(\C)$  is isomorphic to $\partial_\M$ defined in \eqref{def partial M}. 
This result follows from comparing definitions.  Indeed, $\Phi(\M) = b_\M^{-1}\circ a_\M$,
where $a_\M$ and $b_\M$  are defined in \eqref{ZC = ZCM} and \eqref{bM}, and 
$\partial_\M = (\alpha^{-}_\M)^{-1}\circ \alpha^{+}_\M$, where $\alpha^{\pm}_\M$ are defined 
in \eqref{def alpha ind}.  

Thus, it suffices to check commutativity of the following diagrams
\begin{equation}
\label{request of my dear friend}
\xymatrix{
\Z(\C) \ar[r]^(.4){a_\M} & (\C\bt \C^{\opp})^*_\M   \ar[dd]  &  &
\Z(\C) \ar[r]^(.4){b_\M} &  (\C\bt \C^{\opp})^*_\M   \ar[dd] \\
&& \text{and} && \\
\C^\rev \ar[uu] \ar[r]^{\alpha_\M^+} &  \C^*_\M & &  
\C^\rev \ar[uu] \ar[r]^{\alpha_\M^-} &  \C^*_\M,
}
\end{equation}
where  the  arrows $\C^\rev \to \Z(\C)$ are given by the embedding \eqref{embeddings of C to Z(C)}
and the arrows $(\C\bt \C^{\opp})^*_\M \to \C^*_\M$ are given by the restriction of $\C$-bimodule
functors to left $\C$-module functors.  The commutativity is checked directly using definitions
of $\alpha_\M^\pm$ in Section~\ref{alpha} and explicit formulas \eqref{aM computation}  and
\eqref{bM computation} for the $\C$-module functor structures of $a_\M(Z,\, \gamma)$ and 
$b_\M(Z,\, \gamma)$, where $(Z,\,\gamma)$ is an object in $\Z(\C)$.  In the bottom row of
\eqref{request of my dear friend}  we use that $\C^\rev =\C$
as tensor categories. 

Hence,  $\Phi(\M)|_{\C^\rev} = \partial_\M$ in  $\Aut^{br}(\C)= \Aut^{br}(\C^\rev)$. 
\end{proof}

\bre\lb{imp}
Lemma \ref{two inductions} shows in particular that the homomorphism $\partial :  \Pic(\C) \to \Aut(\C)$ defined in \eqref{def partial} factors through $\Pic(\C) \to \Aut^{br}(\C)$.
\ere

The next corollary  was established in \cite{ENO} for braided fusion categories.

\begin{corollary}
Let $\C$ be a factorizable braided tensor category.  Then  $\partial: \Pic(\C) \to \Aut^{br}(\C)$
is an isomorphism. 
\end{corollary}
\begin{proof}
We have $\Z(\C)\cong \C \bt \C^\rev$ and $\Aut^{br}(\Z(\C); \C) =\Aut^{br}(\C^\rev) = \Aut^{br}(\C)$.
\end{proof}

There is canonical action of  $\Aut^{br}(\C)$ on $\Aut^{br}(\Z(\C); \C)$ defined as follows.
Any tensor autoequivalence $g$ of $\C$ induces a braided autoequivalence $\tilde{g}\in \Aut^{br}(\Z(\C))$:
\[
\tilde{g}(Z,\,\gamma) = (g(Z),\, \gamma^g), 
\]
where $(\gamma^g)_X : X \ot g(Z)\xrightarrow{\simeq} g(Z)\ot X$ is given by  $(\gamma^g)_X = g (\gamma_{g^{-1}(X)})$. 

For all $g\in \Aut^{br}(\C)$ and $\alpha \in \Aut^{br}(\Z(\C); \C)$  set
\begin{equation}
\label{conjugation action}
{}^g \alpha := \tilde{g} \circ \alpha \circ \tilde{g}^{-1}. 
\end{equation}
It is clear that ${}^g \alpha$ is trivializable on $\C$, i.e., \eqref{conjugation action} defines the required action. 

\begin{lemma}
\label{eqphi}
The isomorphism $\Phi : \Pic(\C) \xrightarrow{\sim}  \Aut^{br}(\Z(\C); \C)$
is $\Aut^{br}(\C)$-equivariant, i.e., 
\[
\Phi ({}^g \M) = {}^g  \Phi (\M)
\]
for all $g\in \Aut^{br}(\C)$ and $\M \in \Pic(\C)$. 
\end{lemma}
\begin{proof}
This is an immediate consequence of identities 
\[
a_{{}^g\M} = a_{\M} \circ \tilde{g}^{-1} \quad \mbox{and}  \quad b_{{}^g\M} = b_{\M} \circ \tilde{g}^{-1}.
\] 
We have $\Phi({}^g \M) = b_{{}^g\M} ^{-1} \circ a_{{}^g\M} = \tilde{g} \circ  \Phi(\M) \circ  \tilde{g}^{-1} =  {}^g  \Phi (\M)$. 
\end{proof}

\begin{corollary}
\label{Aut crossed module}
The pair of groups $(\Aut^{br}(\C),\,\Aut^{br}(\Z(\C); \C))$  along with the action
\eqref{conjugation action} and homomorphism
$\Sigma: \Aut^{br}(\Z(\C); \C) \to \Aut^{br}(\C)$ from \eqref{Sigma} is a crossed module.
\end{corollary}
\begin{proof}
Follows from Lemmas~\ref{two inductions} and \ref{eqphi}. 
\end{proof}

We will call the crossed module $(\Aut^{br}(\C),\,\Aut^{br}(\Z(\C); \C))$ the {\em autoequivalence} crossed module
of $\C$  and denote it  by $\mathfrak{A}(\C)$.  

\begin{corollary}
\label{P iso A}
The pair of group isomorphisms  $(\id_{\Aut^{br}(\C)},\, \Phi)$ is an isomorphism of crossed modules 
\begin{equation}
\mathfrak{P}(\C)\cong \mathfrak{A}(\C).
\end{equation}
\end{corollary}
\begin{proof}
Follows from Lemmas~\ref{two inductions} and \ref{eqphi}.  
\end{proof}

\subsection{On the kernel and cokernel of  $\partial : \Pic(\C)\to \Aut^{br}(\C)$}

Since the Picard crossed module $\mathfrak{P}(\C)$ is isomorphic to the autoequivalence crossed module of $\mathfrak{A}(\C)$ the kernel of $\partial:\Pic(\C)\to \Aut^{br}(\C)$ is isomorphic to the kernel of the restriction map $\partial:\Aut^{br}(\Z(\C),\C)\to \Aut^{br}(\C)$.

The natural tensor embeddings $\Z_{sym}(\C)\hookrightarrow \C,\C^\rev$ allow us to look at $\C$ and $\C^\rev$ 
as $\Z_{sym}(\C)$-module categories.
The functor \eqref{functor G}: 
\[
\C\boxtimes\C^\rev\to\Z(\C)
\] 
is clearly balanced with respect to these module structures. 
Hence, it factors factors through $\C\boxtimes_{\Z_{sym}(\C)}\C^\rev$. 
Here the tensor product $\C\boxtimes_{\Z_{sym}(\C)}\C^\rev$ of module categories over a symmetric tensor category $\Z_{sym}(\C)$ has a natural structure of braided tensor category, see \cite{DNO}.  The image of $\C\boxtimes_{\Z_{sym}(\C)}\C^\rev$ in $\Z(\C)$
coincides with the full tensor subcategory $\C \vee \C^\rev$ generated by  $\C$ and $\C^\rev$ in $\Z(\C)$. 

\bpr\lb{ke}
The kernel of the restriction map $\partial:\Aut^{br}(\Z(\C),\C)\to \Aut^{br}(\C)$ coincides with the group $\Aut^{br}(\Z(\C),
\C \vee \C^\rev)$ of braided autoequivalences of $\Z(\C)$ trivialisable on $ \C \vee \C^\rev$.
\epr
\bpf
The kernel of the restriction map $\partial:\Aut^{br}(\Z(\C),\C)\to \Aut^{br}(\C)$ coincides with the subgroup $\Aut^{br}(\Z(\C))$ of braided autoequivalences of $\Z(\C)$, trivialisable on both $\C$ and $\C^\rev$. All we need to show is that a braided autoequivalence of $\Z(\C)$, which is trivialisable on both $\C$ and $\C^\rev$ is trivialisable on $\C \vee \C^\rev$.

A braided autoequivalence $F$ of $\Z(\C)$ stabilising both $\C$ and $\C^\rev$ and trivialisable on $\Z_{sym}(\C)$ fits into a commutative diagram:
$$\xymatrix{\C\boxtimes\C^\rev \ar[d]_{F\boxtimes F} \ar[r] \ar@/^12pt/[rr] & \C\boxtimes_{\Z_{sym}(\C)}\C^\rev \ar[d]_{F\boxtimes_{\Z_{sym}(\C)} F} \ar[r] & \Z(\C) \ar[d]_F \\ \C\boxtimes\C^\rev  \ar[r] \ar@/_12pt/[rr] &  \C\boxtimes_{\Z_{sym}(\C)}\C^\rev \ar[r] & \Z(\C)}$$
Thus a braided autoequivalence $F$ of $\Z(\C)$, which is trivialisable on both $\C$ and $\C^\rev$ is also trivialisable on 
$\C \vee \C^\rev$.
\epf

Note that there is a canonical homomorphism
\beq\lb{ker}j: \Pic(\Z_{sym}(\C))\to ker(\Pic(\C) \xrightarrow{\partial} \Aut^{br}(\C;\Z_{sym}(\C))),\eeq
given by the induction of module categories.
Namely, if $\M$ is an invertible $\Z_{sym}(\C)$-module category 
then 
\[
j(\M) =\C\bt_{\Z_{sym}(\C)} \M.
\]
To see that $j(\M)$ is in the kernel of $\partial$, let us take an algebra $A$ in $\Z_{sym}(\C)$ such that
$\M \simeq   \Z_{sym}(\C)_A$.  By Lemma~\ref{upmc}  $j(\M) =\Fun_{\Z_{sym}(\C)}(\C,\, \M) \cong \C_A$. 
The functors $\alpha^\pm_{j(\M)}$ coincide with each other since  $c_{X,A}=c^{-1}_{A,X}$ for all objects $X$ in $\C$,
i.e., $\partial (j(\M))$ is a trivial autoequivalence. 

Let $\Z_{sym}(\C)$ be the symmetric center of $\C$, see Section~\ref{Prelim centralizers}. 
Clearly the restrictions of $\alpha_\M^\pm$ to $\Z_{sym}(\C)$ coincide. 
Hence for an invertible $\M$ the autoequivalence $\partial_\M$ is trivializable on $\Z_{sym}(\C)$, 
i.e., the restriction of $\partial_\M$ to $\Z_{sym}(\C)$ is isomorphic to the identity functor.
Thus the homomorphism \eqref{def partial} factors as follows
$$\Pic(\C)\to \Aut^{br}(\C; \Z_{sym}(\C))\to  \Aut^{br}(\C).$$
Hence, the restriction map defines canonical homomorphism from the cokernel of $\partial$:
\begin{equation}
\label{coc}
coker(\Pic(\C)\xrightarrow{\partial} \Aut^{br}(\C)) \to \Aut^{br}(\Z_{sym}(\C)).
\end{equation}

\section{The Picard crossed module of a pointed braided fusion category}
\label{pointed example}
Let $A$ be a finite abelian group and let $q: A\to k^\times$ be a quadratic form on $A$.
In this Section we explicitly compute  the Picard crossed module of the pointed braided fusion
category $\C := \C(A,\,q)$ associated to the pair $(A,\, q)$ as in Example~\ref{quad form}. 

Note that  
$\C(A,\,q)^\rev\simeq \C(A,\,q^{-1})$. 

\subsection{Invertible module categories over a braided pointed fusion category}
\label{Alexei's parameterization}

The classification of module categories over pointed fusion categories is well known \cite{O2}.
Any indecomposable $\C$-module category $\M$ corresponds to a pair $(B,\, \gamma)$,
where  $B \subset A$ is a subgroup and $\gamma: B \times B \to k^\times$ is a function such that
\begin{equation}
\label{ac}
d(\gamma)(x,\,y,\,z) :=
\frac{\gamma(x+y,\, z) \gamma(x,\, y)}{\gamma(x,\, y+z)\gamma (y,\, z)} =\omega(x,\, y,\, z),\qquad
x,\,y,\, z\in B.
\end{equation}
Here $\omega : A^3\to k^\times$ is the $3$-cocycle defining the associativity constraint of $\C$. 

The pair $(B,\, \gamma)$ is constructed from $\M$ as follows. The simple objects of $\M$
form a transitive $A$-set and $B$ denotes the stabilizer of a point in this set.  The function 
$\gamma: B \times B \to k^\times$ comes from the module associativity constraint of $\M$.
This function  is determined by $\M$ up to a 2-coboundary. 

Let us define a function $\beta: B \times B \to k^\times$ by 
\begin{equation}
\label{def beta}
\beta(x,\, y) =  c(x,\, y) \frac{\gamma(x,\, y)}{\gamma(y,\, x)},\qquad x,\,y\in B. 
\end{equation}

\begin{proposition}
\label{betax}
The function \eqref{def beta} is bimultiplicative and satisfies 
\begin{equation}
\label{rep}
\beta(x,\, x)=q(x), \qquad \mbox{for all }\ x\in B.
\end{equation}
\end{proposition}
\begin{proof} For all $x,\,y,\,z\in B$ we compute
\begin{eqnarray*}
\lefteqn{ \beta(x,\, y+z)
= c(x,\, y+z)\, \frac{\gamma(x,\, y+z)\gamma(y,\,z)}{\gamma(y+z,\,x)\gamma(y,\,z)}  }\\
&=& c(x,\, y+z) \, \frac{\gamma(x+y,\,z)\gamma(x,\,y)}{ \gamma(y,\,z+x)\gamma(z,\,x)}
        \, \omega^{-1}(x,\,y,\,z) \omega^{-1}(y,\,z,\,x) \\
&=& c(x,\, y+z) \, \frac{\gamma(y+x,\,z)\gamma(y,\,x)}{\gamma(y,\, x+z)\gamma(x,\,z)} 
       \,  \frac{\gamma(x,\, y)}{\gamma(y,\, x)} \,\frac{\gamma(x,\, z)}{\gamma(z,\, x)} \,
         \omega^{-1}(x,\,y,\,z) \omega^{-1}(y,\,z,\,x) \\
&=& \beta(x,\,y)\beta(x,\,z)  \, \frac{c(x,\, y+z)}{c(x,\,y)c(x,\,z)} 
        \,\frac{\omega(y,\,x,\,z)}{ \omega(x,\,y,\,z) \omega(y,\,z,\,x) } \\
&=&  \beta(x,\,y)\beta(x,\,z).
\end{eqnarray*}
In the second and the fourth equalities we used identity \eqref{ac}  and in the last equality 
we used \eqref{ab cocycle A}. Thus, $\beta$ is multiplicative in the second argument. 
That it is multiplicative in the first argument is proved in a similar way.  Finally,
the identitty $\beta(x,\,x)= q(x)$ is obtained by setting $y=x$ in \eqref{def beta}. 
\end{proof}

\begin{corollary}
\label{beta bijection}
There is a bijection between 
\begin{equation*} 
 \left\{ \txt{equivalence classes of\\ indecomposable module\\ $\C(A,\,q)$-categories} \right\} 
 \text{ and }   \left\{ \txt{pairs 
 $(B,\,\beta)$, where  $B$ is a subgroup of $A$,  \\  
 $\beta:B\times B \to k^\times$ is bimultiplicative  and \\ $\beta(x,\,x)=q(x),\,x\in B$ } \right\} 
\end{equation*}
\end{corollary} 
\begin{proof} 

Let $B$ be a subgroup of $A$ corresponding to an indecomposable $\C$-module category. 
The formula (\ref{def beta}) defines a map between sets 
\begin{equation}
\lb{map}
\left\{ \txt{ maps $\gamma: B \times B \to k^\times$\\
such that   $d(\gamma) = \omega$\\
modulo coboundaries}  \right\} 
\longrightarrow 
\left\{ \txt{$\beta\in \Hom(B^{\otimes 2},k^\times)$\\ 
such that $\beta(x,\,x)=q(x),\,x\in B$} \right\}.
\end{equation}
We need to  prove that \eqref{map} is a bijection.

Let $\gamma_1,\,\gamma_2$ be $2$-cochains on $B$  such that $d(\gamma_1) = d(\gamma_2) = \omega$ and such that 
\[
c(x,\, y) \frac{\gamma_1(x,\, y)}{\gamma_1(y,\, x)} = c(x,\, y) \frac{\gamma_2(x,\, y)}{\gamma_2(y,\, x)}, \qquad x,\,y\in B
\] 
Then $\dfrac{\gamma_2}{\gamma_1}$ is a symmetric $2$-cocycle on $B$, i.e.,  
$\gamma_1$ and $\gamma_2$ differ by a coboundary.
Thus the map \eqref{map} is injective. 

Consider a diagram
\begin{equation}
\label{H3ab diagram}
\xymatrix{ && H^3_{ab}(A,\,k^\times) \ar[rr] \ar[d]^{res_B} && H^3(A,\,k^\times) \ar[d]^{res_B}\\
\Hom(B^{\otimes 2},\,k^\times) \ar[rr] && H^3_{ab}(B,\,k^\times) \ar[rr] && H^3(B,\,k^\times) }
\end{equation}
with commutative square and the bottom row exact in the middle term.
(The Abelian cohomology groups were defined in Example~\ref{quad form}.)
Let $q$ be a quadratic form on $A$, identified with an element of $H^3_{ab}(A,\,k^\times)$.
It follows from diagram \eqref{H3ab diagram} that $q$ is in the kernel of the composition
$$H^3_{ab}(A,k^\times) \to H^3(A,k^\times)\to H^3(B,k^\times)$$ if and only if  
the restriction of $q$ to $B$ can be represented by some bimultiplicative $\beta:B^{\otimes 2}\to k^\times$. 
This proves surjectivity of (\ref{map}).
\end{proof}

\begin{remark}
Note that the condition \eqref{rep}  along with identity \eqref{cc = sigma} imply
\beq\lb{symb}
\beta(x,\,y)\beta(y,\,x) =  \sigma(x,\,y),\quad x,\,y\in B.
\eeq
\end{remark}

By $\M(B,\, \beta)$ we will denote a module category corresponding to the pair $(B,\beta)$ under the 
bijection from Corollary~\ref{beta bijection}. 

The following Lemma is a special case of the result proved in \cite{N}.

\ble\label{am}
Let  $\M = \M(B,\beta)$ be a $\C(A,\,q)$-module category. Then the group
$\Aut_\C(\M)$ of isomorphism classes of $\C$-module autoequivalences of $\M$ 
fits into a short exact sequence:
$$\xymatrix{ 1\ar[r] & \widehat B \ar[r] & \Aut_\C(\M) \ar[r] & A/B \ar[r] & 1}$$
\ele
\bpf
The homomorphism $\Aut_\C(\M) \to A/B$ assigns the effect of a $\C$-equivalence on the set $A/B$ of simple objects of $\M$. It is clear that this homomorphism is surjective (it is enough to look at the images of $\alpha$-inductions).

The kernel of the homomorphism $\Aut_\C(\M) \to A/B$ consists of isomorphism classes of $\C$-equivalences isomorphic to the identity functor. With a choice of simple object $m\in\M$ a $\C$-module structure on the identity functor on $\M$ gives rise to a character $\psi\in \widehat B$
$$\psi(b)id_m:m = b*m\to b*m = m.$$
It follows from Shapiro's lemma that the character determines the $\C$-module structure. 
\epf

\begin{proposition}
\label{B,beta}
The $\C(A,\,q)$-module category $\M(B,\,\beta)$ is invertible if and only if the form $\beta: B\times B \to k^\times$ 
is non-degenerate.  
\end{proposition}
\begin{proof}
Note that $\M=\M(B,\,\beta)$ is invertible if and only if the $\alpha$-inductions 
\[
\alpha_\M^\pm:\C\to\End_\C(\M)
\]
from Section~\ref{alpha}
induce isomorphisms of groups $A\to \Aut_\C(\M)$ on the level of isomorphism classes of objects. 
We can see that $\alpha$-inductions give morphisms of short exact sequences:
$$
\xymatrix{
0 \ar[r] &B \ar[r] \ar[d] & A \ar[r] \ar[d]^{\alpha_\M^\pm} & A/B \ar@{=}[d]  \ar[r] & 0\\ 
0 \ar[r] & \widehat B \ar[r] &
Aut_\C(\M) \ar[r] & A/B  \ar[r] & 0
},
$$
The homomorphisms $B\to \widehat B$ can be recovered from the $\C$-module functor structures of  $\alpha_\M^\pm(a)$ for $a\in
A$. The $\C$-module functor structure for $\alpha_\M^+(a)$ is given by the diagram:
$$
\xymatrix{ a(bm) \ar[rr]^{\alpha_\M^+(a)_{b,m}} \ar[d]_{\gamma(a,b)} && b(am) \ar[d]^{\gamma(b,a)} \\ (ab)m\ar[rr]^{c(a,b)} && (ba)m }
$$
so that $\alpha_\M^+(a)_{b,m} = \beta(a,b)$ for $a,b\in B$. Here $m$ is a simple object of $\M$. 
Thus, the corresponding homomorphism $B\to \widehat B$ has a form
$b\to \beta(b,\,- )$. 

Similarly, the $\C$-structure for $\alpha_\M^-(a)$ is defined by:
$$\xymatrix{ a(bm) \ar[rr]^{\alpha_\M^-(a)_{b,m}} \ar[d]_{\gamma(a,b)} && b(am) \ar[d]^{\gamma(b,a)} \\ (ab)m \ar[rr]^{c(b,a)^{-1}} && (ba)m }.$$
Hence,  $\alpha_\M^-(a)_{b,m} = \beta(b,a)^{-1}$ for $a,b\in B$ and the corresponding homomorphism $B\to \widehat B$ has a form
$b\to \beta(- ,\,b)^{-1}$.
\end{proof}

From the proof of Proposition~\ref{B,beta} we have the following.
\bco
\label{Alexei's list}
The homomorphism $\partial:Pic(\C(A,q))\to Aut^{br}(A,q)$ sends the class of $\M(B,\beta)$ 
into the unique automorphism $g\in O(A,\,q)$ such that 
\begin{enumerate}
\item[] $g(B)\subset B$,
\item[] $g$ induces the identity on $A/B$, and
\item[] $\beta(b,g(c)) = \beta(c,b)^{-1},\quad \forall b,c\in B$.
\end{enumerate}
\eco

\bre
It follows from \eqref{symb} that the last condition in Corollary~\ref{Alexei's list} can be written as
$\beta(b,g(c)-c) = \sigma(b,c)^{-1},\quad \forall b,c\in B.$ This gives an alternative description of $g$ (cf. \cite{dkr}, the graph of $-g$ is the Lagrangian subgroup $\Gamma(B,\beta)\subset (A\oplus A, q\oplus q^{-1})$ there): 
\begin{enumerate}
\item[] $g(a)-a\in B$  for any $a\in A$ and
\item[] $\beta(b,g(a)-a) = \sigma(b,a)^{-1},\quad \forall b\in B.$
\end{enumerate}

In accordance with the crossed module axiom \eqref{fpcm} 
the map 
\[
\partial: \Pic(\C(A,\,q))\to O(A,\,q)
\]
is $O(A,\,q)$-equivariant: 
$\partial(h(B,\,\beta)) = h\circ\partial(B,\,\beta)\circ h^{-1}$ for $h\in O(A,\,q)$. 
Here $h(B,\,\beta) = (h(B),\,h(\beta))$ with $h(\beta)(a,\,b) = \beta(h^{-1}(a),\,h^{-1}(b))$ and $^{h}{\M(B,\,\beta)}\simeq\M(h(B,\,\beta))$. 

This gives a description of the map $\partial$ for the Picard crossed module $ \mathfrak{P}(\C(A,\,q))$. 
The part which is unclear in this presentation is the group structure of $\Pic(\C(A,\,q))$. 
It appears that the group operation is more accessible on the level of the autoequivalence crossed module 
$\mathfrak{A}(\C(A,\,q))$ (recall that  $\mathfrak{A}(\C(A,\,q)) \simeq  \mathfrak{P}(\C(A,\,q))$
by Corollary~\ref{P iso A}). In the remaining sections we compute this crossed module.
\ere

\subsection{The center of a pointed braided fusion category}

Let $\C =\C(A,\, q)$ be a pointed braided fusion category. The following fact is no doubt
known to experts but we were unable to locate a reference in the literature.

\begin{proposition}
The center $\Z(\C)$ is pointed and $\Z(\C)\cong \C(A \oplus \widehat{A}, Q)$, where
\begin{equation}
\label{Q}
Q(a,\, \phi) = \la \phi,\, a \ra \,  q(a).
\end{equation}
\end{proposition}
\begin{proof}
For any $a\in A$ and $\phi\in \widehat{A}$  there is an invertible object $Z_{a,\phi}$ in $\Z(\C)$
which  is equal to $a$ as an object of $\C$ and has a half braiding given by
\begin{equation}
\label{half br}
c(x,\, a) \, \la \phi,\, x\ra \, \id_{a+x} : 
x \ot  Z_{a,\phi} \xrightarrow{  \sim }  Z_{a,\phi} \ot x,
\end{equation} 
where $c: A^{\times 2} \to k^\times$ is the function \eqref{function c} determining the braiding of $\C$.
That the morphism \eqref{half br} is indeed a central structure on $a$ (i.e., that is satisfies  
necessary coherence conditions) follows from identities \eqref{ab cocycle A} and
\eqref{ab cocycle B}.

Thus, $\Z(\C)$ contains $|A|^2$ non-isomorphic invertible simple objects.
Since the dimension of $\Z(\C)$ is  $\dim(\C)^2 =|A|^2$, the category $\Z(\C)$ is pointed.
Furthermore,  $Z_{a,\phi}  \ot Z_{a',\phi'} = Z_{aa',\phi\phi'},\, a,\,a'\in A,\, \phi,\,\phi'\in \widehat{A}$, 
i.e., the group of invertible objects of $\Z(\C)$ is    $A \oplus \widehat{A}$.  
Finally, from \eqref{half br}  we see that the  braiding on $Z_{a,\phi}  \ot Z_{a,\phi}$ is given by 
the scalar $\la \phi,\, a \ra \, q(a)$.
\end{proof}

\begin{remark}
\label{B}
Let 
\begin{equation}
\label{sigma from q}
\sigma(a,\,b) := \frac{q(a+b)}{q(a)q(b)},\qquad x,y\in A,
\end{equation}
be the bimultiplicative form corresponding to the quadratic
form $q: A \to k^\times$.  Then the  bimultiplicative form corresponding to the form $Q$ defined in \eqref{Q} is
\begin{eqnarray*}
B((a,\, \phi),\, (a',\, \phi')) &=& \frac{Q(a+a',\, \phi+\phi') }{Q(a,\, \phi)  Q(a',\, \phi')} \\
&=&  \la \phi',\,a \ra  \,\la \phi,\,a' \ra \, \sigma(a,\, a'),\qquad a,a'\in A,\, \phi,\phi'\in \widehat{A}. 
\end{eqnarray*}
\end{remark}

\begin{remark}
Note that in general the category $\Z(\Vec_A^\omega)$, where $A$ is an Abelian group
and $\omega \in Z^3(A,\, k^\times)$ is not pointed, see \cite{GMN}.
\end{remark}

Let $\sigma :A\times A \to k^\times$ be the symmetric bimultiplicative form \eqref{sigma from q}.
For any $a \in A$ define  a homomorphism  $\tilde{\sigma}: A \to \widehat{A}$ by
\[ 
\la \tilde{\sigma}(a) ,\, x \ra = \sigma (a,\, x),\qquad \mbox{for all } x\in A.
\]
The embeddings $\C(A,\,q),\, \C(A,\,q)^\rev \hookrightarrow \Z(\C(A,\,q))$  defined in \eqref{embeddings of C to Z(C)} 
are given by injective orthogonal homomorphisms
\begin{eqnarray*}
(A,\, q) \to (A \oplus \widehat{A},\, Q) &:& a \mapsto (a,\, 0), \\
(A,\, q^{-1}) \to (A \oplus \widehat{A},\, Q) &:& a \mapsto (a,\,-\tilde{\sigma}(a)). 
\end{eqnarray*}

\subsection{The Picard group of  $\C(A,\, q)$}
 
By Theorem~\ref{Drinfeld's Theorem} any invertible $\C(A,\, q)$-module category
corresponds to  an orthogonal automorphism $\alpha \in O(A\oplus \widehat{A},\, Q)$
such that $\alpha(a,\,0) = (a,\,0)$ for all $a\in A$.

\begin{proposition}
\label{alpha from f}
Let  $f: \widehat{A} \to A$  be a group homomorphism satisfying the following conditions:
\begin{enumerate}
\item[(i)]  $\id_{\widehat{A}}-\tilde{\sigma} f$ is invertible,
\item[(ii)] $\la \phi,\, f(\phi)\ra = q( f(\phi))$ for all $\phi \in \widehat{A}$.
\end{enumerate}
Then the map
\begin{equation}
\label{our alpha}
\alpha_f(a,\, \phi) = (a + f(\phi),\, \phi - \tilde{\sigma} f(\phi) ),\qquad a\in A,\, \phi\in \widehat{A}.
\end{equation}
is an orthogonal automorphism  of $(A\oplus \widehat{A},\, Q)$ that restricts to
the identity on $A$.

Conversely, any orthogonal automorphism  with this property is of the form \eqref{our alpha}
for a unique homomorphism $f: \widehat{A} \to A$ satisfying  conditions (i) and (ii).
\end{proposition}
\begin{proof}
Suppose a group homomorphism $f: \widehat{A} \to A$ is given.
Clearly, $\alpha_f$ is a homomorphism and its restriction to $A$ is the identity. 
Condition (i) in the statement of the Proposition
is equivalent to $\alpha_f$ being invertible. Let us explore the property of
$\alpha_f$ being orthogonal. We compute
\begin{eqnarray*}
Q(\alpha_f(a,\, \phi))
&=& Q(a + f(\phi) ,\, \phi - \tilde{\sigma} f(\phi)) \\
&=& Q(a,\, \phi) \, Q(f(\phi),\,-\tilde{\sigma} f(\phi)) \,B((a,\, \phi),\, (f(\phi),\,-\tilde{\sigma} f(\phi))) \\
&=& Q(a,\, \phi) \, \sigma( f(\phi) ,\, f(\phi))^{-1} \,   q(f(\phi))  \, \la \phi,\,f(\phi)\ra \, \\
&=& Q(a,\, \phi) \,   q(f(\phi))^{-1} \, \la \phi,\,f(\phi) \ra ,
\end{eqnarray*}
whence $\alpha_f$ is orthogonal if and only if condition (ii) is satisfied.

Let us prove the converse statement.
Let $\alpha \in O(A\oplus \widehat{A},\, Q)$ be such that $\alpha$ restricts to the identity on $A$. 
Let $f: \widehat{A} \to A$  and $g: \widehat{A} \to \widehat{A}$  be homomorphisms
such that $\alpha(0,\, \phi) = (f(\phi),\, g(\phi))$ for all $\phi\in \widehat{A}$. 
Since $\alpha$ preserves the quadratic form $Q$ 
the condition $Q(0,\, \phi)=1$ implies $Q(f(\phi),\,g(\phi))=1$ which is equivalent to
\begin{equation}
\label{la fg ra} 
\la g(\phi),\, f(\phi)\ra \, q (f(\phi)) =1.
\end{equation}

Next, for arbitrary $a\in A$ and $\phi\in \widehat{A}$ we have 
\begin{equation}
\label{alpha a phi}
\alpha(a,\, \phi) = (a+f(\phi),\,g(\phi)).
\end{equation}
We have $Q(\alpha(a,\, \phi)) = Q(a,\, \phi) = \la \phi,\, a\ra \, q(a)$.
On the other hand, we compute
\begin{eqnarray*}
Q(\alpha(a,\, \phi))
&=& Q(a+f(\phi),\,g(\phi))  \\
&=& Q(a,\,1) \, Q(f(\phi),\, g(\phi))\,  \la g(\phi),\, a \ra \, \sigma(a, \, f(\phi)) \\ 
&=& q(a)\, \la g(\phi),\, a \ra \, \sigma ( f(\phi),\, a) \\
&=& q(a) \,\la g(\phi) + \tilde{\sigma} f(\phi),\, a \ra.
\end{eqnarray*}
Comparing two expressions we obtain  
\begin{equation}
\label{g phi}
g(\phi)=\phi - \tilde{\sigma} f(\phi), \qquad \mbox{for all }\phi\in \widehat{A}.
\end{equation}
This along with \eqref{alpha a phi} yields \eqref{our alpha}.

Substituting \eqref{g phi} to  \eqref{la fg ra} we obtain
\[
\la \phi,\, f(\phi)\ra  \,q (f(\phi)) = \la   \tilde{\sigma}f(\phi) , \,f(\phi)\ra  = \sigma ({f(\phi)}, \,f(\phi)) = q (f(\phi))^2,
\]
whence $\la \phi,\, f(\phi)\ra   = q (f(\phi))$ as required.
\end{proof}

Let $P(A,\, q)$ be the set of group homomorphisms $f: \widehat{A}\to A$  satisfying conditions
(i) and (ii) of Proposition~\ref{alpha from f}, i.e.,
\begin{equation}
\label{introducing P(a,q)}
P(A,\,q) := \left\{
\txt{homomorphisms $f: \widehat{A}\to A$ such that\\  $\id_{\widehat{A}}-\tilde{\sigma}\circ f$ is invertible
and\\ $\la \phi,\, f(\phi)\ra = q( f(\phi))$ for all $\phi \in \widehat{A}$
}
\right\}.
\end{equation}

Endow the set $P(A,\, q)$ with the following binary operation
\begin{equation}
\label{diamond}
f \diamond g  = f + g - f\circ \tilde{\sigma}\circ g,\qquad f,\,g\in P(A,\, q).
\end{equation}

\begin{proposition}
\label{P(A,q) is a group}
The set $P(A,\,q)$ with the  operation $\diamond$ defined in \eqref{diamond} is a group.
Furthermore, the map
\begin{equation}
\label{f mapsto alpha}
f \mapsto \alpha_f : P(A,\,q)\to \Aut^{br}\big(\Z(\C(A,\,q)),\C(A,\,q)\big),
\end{equation}
where  $\alpha_f\in \Aut^{br}(\Z(\C(A,\,q)))$
is defined in \eqref{our alpha}, is a group isomorphism.
\end{proposition}
\begin{proof}
By Proposition~\ref{alpha from f}  the assignment  \eqref{f mapsto alpha}
is a bijection. Since
\[
\alpha_f\circ \alpha_g
=\alpha_{f\diamond g},\qquad \mbox{for all } f,\,g \in P(A,\,q),
\]
we see that $P(A,\,q)$ is a group and the assignment 
\eqref{f mapsto alpha}  is a group isomorphism. 
\end{proof}

\begin{remark}
\label{inverse formula}
Clearly, the identity element of $P(A,\,q)$ is the zero homomorphism. 
Let us describe the inverse of $f\in P(A,\, q)$. 

It is immediate from  \eqref{diamond} that the inverse of $f$ 
is given by the formula 
\begin{equation}
\label{f inverse}
f^{-1}= (f \circ \tilde\sigma-\id_A)^{-1} \circ f.
\end{equation}
Let $f^* : \widehat{A}\to A$ denote the homomorphism
dual to $f$.  We claim that $f^*\in P(A,\,q)$ and that $f^*$ is the inverse
of $f$ with respect to the multiplication 
$\diamond$.  Indeed, equality of quadratic forms in condition (ii) of Proposition~\ref{alpha from f} 
implies equality of the corresponding bilinear forms:
\[
\la f+f^*(\phi),\, \psi \ra = \sigma(f(\phi),\, f(\psi)),\qquad \phi,\,\psi\in \widehat{A}, 
\]
whence $f+f^* =f^*\circ \tilde\sigma \circ f$,  i.e.,  $f^*$  coincides with
the right hand side of \eqref{f inverse}.
\end{remark}

\begin{corollary}
\label{P iso P}
There is a group isomorphism $P(A,\,q)\cong \Pic(\C(A,\,q))$.
\end{corollary}
\begin{proof}
This follows from Proposition~\ref{P(A,q) is a group} and Theorem~\ref{Drinfeld's Theorem}.
\end{proof}

\begin{remark}
\label{comparing 2 parameterizations}
We have two parameterization for the group $\Pic(\C(A,\,q))$.  The first one is given in terms
of pairs $(B,\, \beta)$, where $B \subset A$ is a subgroup and $\beta: B \times B \to k^\times$
is a non-degenerate bimultiplicative map such that $\beta(x,\,x) =q(x)$ for all $x\in B$, see Corollary~\ref{beta bijection}
and Proposition~\ref{B,beta}. The second one  is given in terms of the set $P(A,\,q)$
consisting of homomorphisms $f: \widehat{A}\to A$ satisfying conditions listed in
\eqref{introducing P(a,q)}.

Let us establish a bijection between  these parameterizations. 
Let $\M =\M(B,\, \beta)$ denote the invertible $\C(A,\,q)$-module
category corresponding to a pair $(B,\, \beta)$ as above.  Let $\Phi(\M)$ denote the corresponding
braided autoequivalence of $\Z(\C(A,\,q)$ defined as in \eqref{PhiM}.  By Proposition~\ref{alpha from f}
$\Phi(\M)=\alpha_f$ for a unique $f \in P(A,\,q)$. Let $\phi\in \widehat{A}$ and let $b =f(\phi)$.
Then $b$ is uniquely determined by the condition
\[
\Phi(\M)(Z_{0,\phi}) = Z_{b,\psi} \quad \text{ for some $\psi\in \widehat{A}$}.
\]
Equivalently,
\[
a_\M (Z_{0,\phi}) = b_\M(Z_{b,\psi}),
\]
where $a_\M$ and $b_\M$ are functors defined in \eqref{ZC = ZCM} nd \eqref{bM}.
Note that $b\in B$ since  the functor $a_\M (Z_{0,\phi})$ is identical on the classes of isomorphic
objects of $\M$. 

Take $x\in B$ and compare isomorphisms
\begin{eqnarray}
\label{aMZ0}
x \ot a_\M(Z_{0,\phi})(?) & \xrightarrow{\simeq}&  a_\M(Z_{0,\phi})(x\ot ?) \quad \text{and}\\
\label{bMZ0}
x \ot b_\M(Z_{b,\psi})(?) & \xrightarrow{\simeq}&  b_\M(Z_{b,\psi})(x\ot ?)
\end{eqnarray}
coming from  left $\C(A,\,q)$-module functor structures of $a_\M(Z_{0,\phi})$ and
$b_\M(Z_{0,\phi})$. 

Using equations \eqref{aM computation} and  \eqref{half br} we see that the isomorphism
\eqref{aMZ0} is given by
\begin{equation}
\label{aMZ0 explicit}
x \ot (Z_{0,\phi} \ot ?) \xrightarrow{ \langle \phi,\,x\rangle} Z_{0,\phi} \ot (x\ot ?).
\end{equation}

On the other hand, using equation \eqref{bM computation} we see that the isomorphism \eqref{bMZ0}
is given by
\begin{equation}
\label{bMZ0 explicit}
x \ot (Z_{b,\psi} \ot ?) \xrightarrow{\gamma(x,\,b)}
(x \ot Z_{b,\psi}) \ot ? \xrightarrow{c(x,\,b)}
(Z_{b,\psi}\ot x) \ot ? \xrightarrow{\gamma(b,\,x)^{-1}}
Z_{b,\psi}\ot (x \ot ?),
\end{equation}
where $\gamma: B\times B \to k^\times$ is the function that determines the module associativity
of $\M(B,\, \beta)$, see \eqref{ac}, and $c: A \times A \to k^\times$ is the braiding of $\C(A,\,q)$.
From \eqref{def beta} we see that the product of scalars in the composition \eqref{bMZ0 explicit}
is equal to $\beta(x,\,b)$. Since $\beta$ is non-degenerate it follows that $b= f(\phi)$
is completely  determined by the condition
\[
\langle \phi,\, x \rangle =\beta(x,\, b).
\]
Thus, the homomorphism $f: \widehat{A}\to A$ corresponding to $(B,\, \beta)$ is given by the composition
\begin{equation}
\label{ f from B, beta}
f: \widehat{A} \to \widehat{B} \xrightarrow{\hat{\beta}} B \hookrightarrow A,
\end{equation}
where  $\widehat{A} \to \widehat{B}$ is the surjection dual to the embedding $B \hookrightarrow A$  and
$\hat{\beta}: \widehat{B} \xrightarrow{\sim} B$ is the isomorphism induced by $\beta$.
\end{remark}

\begin{example}
\begin{enumerate}
\item[(i)]
Suppose  $q$ is non-degenerate (i.e., the category $\C(A,\, q)$ is non-degenerate).  Then 
$\tilde{\sigma}: A \to \widehat{A}$ is an isomorphism and the map  
\[
P(A,\, q) \to O(A,\,q) :
f \mapsto \id_A -  f\circ \tilde\sigma
\]
is an isomorphism. 
\item[(ii)]
Suppose $q=1$ (i.e., the category $\C(A,\, q)$ is  Tannakian). Then 
\[
P(A,\,q) =\{ \phi:  \widehat{A}\to A \mid \la \phi,\, f(\phi)\ra =1\}.
\] 
Thus, elements of  $P(A,\,q)$ are identified with alternating bimultiplicative maps
$\widehat{A} \times \widehat{A} \to k^\times$ and
\[
P(A,\,q) \cong \wedge^2 A  \cong H^2(\widehat{A},\,k^\times),
\]
cf. \cite[Corollary 3.17]{ENO}. 
\item[(iii)]
Suppose that $\sigma =1$ but $q\neq 1$ (i.e., the category $\C(A,\, q)$ is  symmetric but
not Tannakian). In this case $q\in \widehat{A}$ is a character of order $2$. Let $\langle q \rangle$
denote the subgroup of $\widehat{A}$ generated by $q$. We have
\[
P(A,\,q) \cong  
\begin{cases}
H^2(\widehat{A},\,k^\times) & \text{if  $\langle q \rangle$ is a direct summand in $\widehat{A}$}, \\ 
H^2(\widehat{A},\,k^\times) \times \mathbb{Z}/2\mathbb{Z} & \text{otherwise}. 
\end{cases}
\] 
This agrees with the result of \cite{C} in the case of semisimple Hopf algebras.
\end{enumerate}
\end{example}

\subsection{Description of the Picard crossed module of $\C(A,\,q)$}

Let $\C(A,\,q)$ be a pointed braided fusion category.  By Corollary~\ref{P iso A}
the Picard crossed module of $\C$ is isomorphic to the autoequivalence  crossed module
\[
\mathfrak{A}(\C(A,\,q)) = \big(\Aut^{br}(\Z(\C(A,\,q)); \C(A,\,q) ) ,\,\Aut^{br}(\C(A,\,q))\big) \cong \big(P(A,\,q),
O(A,\,q)\big).
\]
introduced in Section~\ref{def Pic cr mod}. 

By Lemma~\ref{two inductions} the structural homomorphism
\begin{equation}
\label{alpha ind}
\partial: \Pic(\C(A,\, q)) \cong \Aut^{br}(\Z(\C(A,\,q)); \C(A,\,q) )  \to \Aut^{br}(\C(A,\,q)). 
\end{equation}
is given by  restriction of autoequivalences in $\Aut^{br}(\Z(\C(A,\,q));\C(A,\,q))$ 
to $\C(A,\,q)^\rev \subset \Z(\C(A,\,q))$.

Let us describe $\partial$ explicitly.
We already observed that the tensor  subcategory $\C(A,\,q)^\rev
\subset \Z(\C(A,\,q))$  corresponds to the subgroup
$\{(a,\, -\widehat{a}) \mid a\in A \} \subset A \oplus \widehat{A}$.
Given $f \in P(A,\,q)$ we have
\[
\alpha_f(a,\,-\tilde{\sigma}(a)) = (a - f \tilde{\sigma}(a),\,  -(\tilde{\sigma}(a) - \tilde{\sigma}f\tilde{\sigma}(a)).
\]
Hence, 
\begin{equation}
\label{partial}
\partial({f}) = \id_A - f\circ \tilde{\sigma},\quad f\in P(A,\,q).
\end{equation}
Next, for any $g\in O(A,\, q)$ let $\tilde{g}\in O(A\oplus\widehat{A},\,Q)$ be the orthogonal
automorphism induced by $g$, i.e., $\tilde{g}(a,\, \phi) = (g(a),\, \phi\circ g^{-1})$. 
It is straightforward to check the identity
\[
\tilde{g} \circ \alpha_f \circ \tilde{g}^{-1} =\alpha_{({}^g f)},
\]
where \begin{equation}
\label{action}
{}^g f = g\circ f \circ g^{-1}  ,\quad g\in O(A,\,q), \, f\in P(A,\,q).
\end{equation}
Thus, the autoequivalence crossed module of $\C(A,\, q)$ is
\[
\mathfrak{A}(\C(A,\,q)) \simeq (P(A,\, q),\, O(A,\,q))
\]
with structural operations \eqref{partial} and \eqref{action}. 

\subsection{Invariants of $\mathfrak{P}(\C(A,\,q))$}

The kernel and  the cokernel of the homomorphism $\partial$ are important invariants of 
a crossed module. Below we compute the kernel of $\partial$ for the crossed module $\mathfrak{P}(\C(A,\,q))$. We also describe the cokernel of $\partial$ for the crossed module $\mathfrak{P}(\C(A,\,q))$ when $\Z_{sym}(\C(A,\,q))$ is Tannakian. 

As before $A^\perp\subset A$ denotes the kernel of $\sigma$. Note that  
$\C(A^\perp,\, q|_{A^\perp}) = \Z_{sym}(\C(A,\,q))$ is a symmetric fusion category. 

\begin{proposition}
\label{Ker partial}
The group homomorphism \eqref{ker} $$j: \Pic(\C(A^\perp,\, q|_{A^\perp}))  \to  ker(\partial)$$ is an isomorphism. 
\end{proposition}
\begin{proof}
The homomorphism $j$ can be explicitly described as follows. 
For $g\in P(A^\perp,\, q|_{A^\perp})$ the image  $j(g)\in P(A,\, q)$ 
is the composition
\[
j(g) : \widehat{A}\to \widehat{A^\perp} \xrightarrow{g}  A^\perp \hookrightarrow A,
\]
where the first arrow is the restriction of a character and the last arrow is the embedding. 

We will construct the inverse homomorphisms
\[
i: \mbox{Ker}(\partial) \to  \Pic(\C(A^\perp,\, q|_{A^\perp}))  
\] to $j$.
Let  $f \in \mbox{Ker}(\partial)$. Then $f \circ  \tilde{\sigma}=0$, i.e., $f|_{\widehat{A/A^\perp}} =0$.
By Remark~\ref{inverse formula} we also have $f^*\in \mbox{Ker}(\partial)$ and, hence,  $f^* \circ  \tilde{\sigma}=0$.
Taking the dual we get , $\tilde{\sigma} \circ f =0$, i.e., $f(\widehat{A}) \subset A^\perp$.  
Hence $f$ descends to a homomorphism
\[
i(f): \widehat{A^\perp} \cong  \widehat{A} /  \widehat{(A/A^\perp)}  \to A^\perp, 
\]
which is easily seen to be in $P(A^\perp,\, q|_{A^\perp})$. 
\end{proof}

Now let $\C(A,\,q))$ be a pointed category 
whose symmetric center $\Z_{sym}(\C(A,\,q))$ is Tannakian.
 In other words let $q|_{A^\perp}=1$. 
 Note that in this case the form $q$ descents to $A/A^\perp$ (we denote the descendent form by $\tilde q$). 
Below we describe the kernel of the homomorphism \eqref{coc} for $\C(A,\,q)$.
\bpr
Let $q|_{A^\perp}=1$. Then the kernel of the canonical homomorphism \eqref{coc} 
$$coker\big(\Pic(\C(A,\, q)) \xrightarrow{\partial} \Aut^{br}(\C(A,\,q))\big)\to Aut(A^\perp)$$ 
is isomorphic to the abelian group $Hom(A/A^\perp,A^\perp)$. In other words, the cokernel 
$C=coker\big(\Pic(\C(A,\, q)) \xrightarrow{\partial} \Aut^{br}(\C(A,\,q))\big)$ fits into an exact sequence
\beq\lb{ext}\xymatrix{0\ar[r] & Hom(A/A^\perp,A^\perp)\ar[r] & C \ar[r] & Aut(A^\perp)}.\eeq
\epr
\bpf
From the commutativity of the diagram
$$
\xymatrix{P(A,\,q)\ar[r]^\partial \ar[d] & O(A,\,q)\ar[d]\\ P(A/A^\perp,\,\tilde q)\ar[r]_{\partial}^\simeq & O(A/A^\perp,\,\tilde q)}
$$
it follows that $coker\big(P(A,\, q) \xrightarrow{\partial} O(A,\,q)\big)$ coincides with 
$$ker\big(O(A,\,q)\to O(A/A^\perp,\,\tilde q)\big)/im(\partial)\cap ker\big(O(A,\,q)\to O(A/A^\perp,\,\tilde q)\big).$$
Now $ker\big(O(A,\,q)\to O(A/A^\perp,\,\tilde q)\big)$ consists of automorphisms of the form $id_A + \phi$ for $\phi\in Hom(A,A^\perp)$. Indeed any element of $ker\big(O(A,\,q)\to O(A/A^\perp,\,\tilde q)\big)$ must have this form and conversely any automorphisms of this form preserves $q$:
$$q(a +\phi(a)) = q(a)q(\phi(a))\sigma(a,\phi(a)) = q(a).$$
Note that composition of automorphisms induces the following group operation on $Hom(A,A^\perp)$:
$$\phi*\psi = \phi+\psi+\phi\circ\psi.$$
It is straightforward that $C=\{\phi\in Hom(A,A^\perp)|\ id_A + \phi$ is invertible $\}$ with the group operation $*$ fits into an exact sequence \eqref{ext}.

All we need to show now is that the intersection $/im(\partial)\cap ker\big(O(A,\,q)\to O(A/A^\perp,\,\tilde q)\big)$ is trivial. Assume that $\partial(f) = id_A + \phi$ for $\phi\in Hom(A,A^\perp)$. Then $\phi = -f\circ\tilde\sigma$ so that $im(f)\subset A^\perp$. We also have $\partial(f^*) = id_A + \psi$ for $\psi\in Hom(A,A^\perp)$, which implies that $im(f^*)\subset A^\perp$. Then 
$$\phi = -f\circ\tilde\sigma = -(\tilde\sigma\circ f^*)^* = 0.$$
\epf

\bibliographystyle{ams-alpha}

\end{document}